\documentclass[a4paper,12pt]{amsart}
\usepackage{amsmath,amsthm,amsfonts,amssymb,stmaryrd,mathrsfs,enumitem}
\usepackage{latexsym}
\usepackage{fullpage}
\usepackage{a4,color,palatino,fancyhdr}
\usepackage{graphicx}
\usepackage{float}  
\usepackage[small, bf, margin=90pt, tableposition=bottom]{caption}
\RequirePackage{amssymb}
\RequirePackage[T1]{fontenc}
\usepackage{amscd}
\usepackage{xypic}
\input{xy}
\xyoption{all}
\xyoption{poly}
\usepackage[all]{xy}
\usepackage{tikz}
\usepackage{cases}
\usepackage{verbatim}
\newtheorem{thm}{Theorem}[section]
 \newtheorem{cor}[thm]{Corollary}
 \newtheorem{lem}{Lemma}[section]
 
 \theoremstyle{definition}
 \newtheorem{defn}{Definition}[section]
 
 \theoremstyle{remark}
 \newtheorem{rem}{Remark}
 \newtheorem{example}{Example}[section]
 
 \numberwithin{equation}{section}

 \newcommand{\id}{\mathrm{id}}

\begin{document}

\newcommand{\auths}[1]{\textrm{#1},}
\newcommand{\artTitle}[1]{\textsl{#1},}
\newcommand{\jTitle}[1]{\textrm{#1}}
\newcommand{\Vol}[1]{\textbf{#1}}
\newcommand{\Year}[1]{\textrm{(#1)}}
\newcommand{\Pages}[1]{\textrm{#1}}
\newcommand{\RNum}[1]{\uppercase\expandafter{\romannumeral #1\relax}}

\title[Orbit configuration spaces of standard action]
 {Orbit configuration space of standard action and cellular methods of poset}

\email{}

\author{Junda Chen}
\address{School of Mathematical Sciences, East China Normal University, Shanghai,   China}
\address{Center for Topology and Geometry based Technology, Hebei Normal University,
	Shijiazhuang, China}
\email{jdchen@math.ecnu.edu.cn}
\thanks{The author is supported by \textit{Center for Topology and Geometry based Technology} of Hebei Normal University and Natural Science Foundation of China (NSFC grant no. 11971144) and High-level Scientific Research Foundation of Hebei Province.
}
\begin{abstract}
We extend  the existing idea of "cellular poset", introduce a collection of "cellular methods" for the computation of homology of intersection lattice of a complicated subspace arrangement, and for the computation of multiplicative structure induced by intersection. As an application, we give a presentation of cohomology ring of orbit configuration space of standard action by cellular methods and a spectral sequence associated with Grothendieck fibration of poset.
\end{abstract}

\maketitle
\section{Introduction}

Cohomology of complements of subspace arrangements is a well studied field. The additive structure is originally given by Goresky and MacPherson in \cite{Goresky1988} using stratified Morse theory, the result has also been proved in \cite{Jewell1994,Jewell1994a,Vassiliev1992} by some different approach. The ring structure is studied in \cite{DeConcini1995,Yuzvinsky2002,Deligne2000,DeLongueville2001,Chen2020}. To computing the cohomology ring of complement, we need to \begin{itemize}
	\item Compute the homology of lower intervals of intersection lattice.
	\item Compute the multiplicative structure of homology induced by intersection of subspace.
\end{itemize}
See \cite{DeLongueville2001} for more details. Usually, we need to make a huge effort to obtain explicit presentation of cohomology ring of complement of a complicated arrangement. For example, the cohomology ring of "k-equal" manifold is given by Yuzvinsky in \cite{Yuzvinsky2002}.

\vskip .2cm

In this paper, we introduce some techniques, especially the "cellular methods" extending the idea in \cite{Everitt2015}, which are useful for computing the homology of complicated intersection lattice and multiplicative structure on it, we study a specific subspace arrangement $\mathcal{A}$ , whose complement is orbit configuration space of standard action. A standard $\mathbb{Z}_k^m$-action on $\mathbb{C}^m$ is a map $\varphi : \mathbb{Z}_k^m \times \mathbb{C}^m \rightarrow \mathbb{C}^m$ such that $$\varphi((z_1,z_2,...), (x_1,x_2,...))=(e^{2\pi iz_1/k}x_1, e^{2\pi iz_2/k}x_2,...)$$
Orbit configuration space $F_G(X,n)$ is defined by Xicotncatl \cite{Xicotncatl} that $$F_G(X,n)= \{(x_1,x_2,...,x_n)\in X^n | Gx_i \cap Gx_j = \emptyset \text{ for } i\neq j\}$$
\begin{defn}
	Let $M= \underbrace{\mathbb{C}^m\times...\times \mathbb{C}^m}_n$ and 
	$\mathcal{A}=\{V_{i,j,g}\}$ be a subspace arrangement in $M$ where $$V_{i,j,g}=\{(x_1,...,x_n)\in M | gx_i=x_j\} $$ for $i\neq j,g\in \mathbb{Z}_k^m$. Let $\mathcal{P}$ denote the intersection lattice of $\mathcal{A}$.
\end{defn}

It's obvious that orbit configuration space $F_{\mathbb{Z}_k^m}(\mathbb{C}^m,n)$ of standard action is complement of this arrangement $\mathcal{A}$.

\vskip .2cm
Many known results \cite{Casto2016,DENHAM2018,Feichtner2002,Bibby2018} about orbit configuration space rely on the assumption that the group action is free or almost free, the orbit configuration of non-free group action is partially discussed in \cite{Chen2021}.
In this paper, we give a presentation of cohomology ring $H^*(F_{\mathbb{Z}_k^m}(\mathbb{C}^m,n))$. 
The computation is based on some algebraic techniques which contains
\begin{itemize}
	\item A spectral sequence associated with Grothendieck fibration of join semilattice, see Section 3. 
	\item Cellular methods of graded poset, see Section 4.
\end{itemize}

Roughly speaking, the intersection lattice $\mathcal{P}$ is closely related with a Grothendieck fibration $L_k^m$ on partition lattice $\Pi_n$, then  the spectral sequence be applied and the fiber and base are both "cellular". 

The cohomological \textit{cellular poset} is defined by Everitt and
Turner in \cite{Everitt2015}, we extend their idea and give the definition of homological cellular for a pair of poset $P$ and a copresheaf $\mathcal{G}$ on it. We also define "cellular form" $\varLambda(P,\mathcal{G})$ of pair $(P,\mathcal{G})$ by it's property, which can be checked in polynomial time and it is the core of our cellular methods. See Section 4 for more details.

Using the cellular methods, we give the multiplicative structure of homology of intersection lattice by "construct" morphism of cellular forms rather than "compute" it, see Section 5. Actually, these techniques have expanded far beyond the application in this paper. It seems useful when consider more general cases, for example, orbit configuration space of reflection group.

\vskip .2cm

Before the main results of this paper, we need some definitions.

\begin{defn}
	Let $L=\Pi_n$ be the partition lattice of $[n]$. For each partition $I\in L$, let $I'$ be the set of blocks that contains at least two points. A $\mathbb{Z}_k$-coloring of a block $p\in I'$ is an \textbf{equivalence class} $\phi$ of functions $f: p \to \mathbb{Z}_k$ with relation $gf \sim f$ for all $g\in \mathbb{Z}_k$. 

	Let $\mathcal{C}_I^m$ be the set of partial matrix (matrix with some undefined entry "$?$"), looks like
	$$\theta=
	\bordermatrix{
		& 1 & 2 & ...& m \cr
		p_0 & \phi_{1} & ? & ...& \phi_{2} \cr
		p_1 & ? & \phi_{3} & ... & ? \cr
		... & ... & ... & ... & ...\cr
		p_s & ? & \phi_{4} &...& \phi_{5}\cr
	}$$
	indexed by $I'\times [m]$ and every entry $\theta_{p,t}$ is a $\mathbb{Z}_k$-coloring of block $p$ or an undefined "$?$". 
	
	Define \begin{itemize}
		\item $\text{Ud}(\theta)$ be the set of index $(p,t)\in I'\times [m]$ such that $\theta_{p,t}$ is undefined.
		\item $L_k^m = \sqcup_{I \in L}\mathcal{C}_I^m$
		\item $\pi : L_k^m \to L$ be the natural projection.
	\end{itemize}
	
\end{defn}

\begin{defn}
	Assume $\theta \in \mathcal{C}_I^m$. An $l$-completion of $\theta$ is made by changing $l$ undefined entries to a coloring of its row index. We say it is \textbf{completion} of $\theta$ if there is no "$?$" left. Then $\mathcal{C}_I^m$ becomes a poset with relation $\theta_1 \leq \theta_2$ iff $\theta_1$ is a $l$-completion of $\theta_2$ for some natural number $l$. Let $r_f(\theta)$ be the number of "$?$" in $\theta$ and $r_b(\theta)$ be the rank of $I$ in geometric lattice $L=\Pi_n$ for any $\theta \in \mathcal{C}_I^m$. It's obvious that $\mathcal{C}_I^m$ is graded poset with rank function $r_f$. 
\end{defn}
\begin{rem}

We will illustrate in Section 5 that there is a well defined join semilattice structure on $L_k^m$, compatible with partial order on $\mathcal{C}_I^m$, and the projection map $\pi: L_k^m \to L$ is a Grothendieck fibration of join semilattice. We call two elements $\theta  \in \mathcal{C}_{I}^m, \psi  \in \mathcal{C}_{J}^m$ are \textbf{\textit{independent}} in $L_k^m$ if $r_b(\theta)+r_b(\psi)=r_b(\theta\vee\psi)$ and $r_f(\theta)+r_f(\psi)=r_f(\theta\vee\psi)$.
\end{rem}
\begin{defn}\label{defcp}
	Let $\theta \in \mathcal{C}_{I}^m$,  denote $\text{Cp}(\theta)$ the free abelian group generated by all completions $\{\eta_1,\eta_2,...\}$ of $\theta$. Denote $\text{BCp}(\theta)$ the subgroup of all formal sum $\sum_{i}k_{i} \eta_i$ in $\text{Cp}(\theta)$ satisfying
	$$\sum_{\eta_i <\nu} k_{i}=0$$ for any  $\nu\leq\theta$ with rank $r_f(\nu)=1$.
	
\end{defn}


\begin{thm}\label{main}
    For $m>1$, $H^*(F_{\mathbb{Z}_k^m}( \mathbb{C}^m, n))$ is a $L_k^m$-graded $\mathbb{Z}$-algebra, i.e.
    $$H^*(F_{\mathbb{Z}_k^m}( \mathbb{C}^m, n)) = \bigoplus_{\theta\in L_k^m}H^*(F_{\mathbb{Z}_k^m}( \mathbb{C}^m, n))_{\theta}$$ and every piece $H^*(F_{\mathbb{Z}_k^m}( \mathbb{C}^m, \Gamma))_{\theta}$ is a free abelian group $A^*(L)_{\pi(\theta)}\otimes \text{BCp}(\theta)$, has rank 
    \begin{equation}\label{eq of rank}
    	|\mu(\hat{0},\pi(\theta))|\prod_{(p,t)\in \text{Ud}(\theta)} (k^{|p|-1}-1)
    \end{equation}
    where $A^*(L)$ is OS-algebra of partition lattice $L=\Pi_n$ (see \cite{Yuzvinsky2001} or \cite{Dimca2009} for OS-algebra), $A^*(L)_{\pi(\theta)}$ is the piece of grading $\pi(\theta)$, $\mu(-,-)$ is M\"{o}bius function.
    Every element in \\$H^*(F_{\mathbb{Z}_k^m}( \mathbb{C}^m, \Gamma))_{\theta}$ has degree $(2m-1)r_b(\theta)+r_f(\theta)$. 
    
    For any two elements \begin{align*}
    	x=a\otimes\sum k_i\eta_i \in H^*(F_{\mathbb{Z}_k^m}( \mathbb{C}^m, \Gamma))_{\theta} \\y=b\otimes\sum s_i\nu_i \in H^*(F_{\mathbb{Z}_k^m}( \mathbb{C}^m, \Gamma))_{\psi}
    \end{align*}
    the cup product $x\cup y$ is contained in $H^*(F_{\mathbb{Z}_k^m}( \mathbb{C}^m, \Gamma))_{\theta\vee \psi}$, given by 
    \begin{equation}\label{eq of prod}
     x\cup y =\begin{cases}
    	\text{sgn}(|\theta,\psi|)(a\cdot b) \otimes \sum_{ij}k_is_j\eta_i\vee\nu_j  & \text{ for independent } \theta,\psi\\
    	0 & \text{ otherwise }	
    \end{cases}\end{equation}
    where $a\in A^*(L)_{\pi(\theta)},b \in A^*(L)_{\pi(\psi)},\sum k_i\eta_i\in \text{BCp}(\theta), \sum s_i\nu_i \in \text{BCp}(\psi)$, $\vee$ is the join operation in $L_k^m$, $a\cdot b$ is the product in OS-algebra $A^*(L)$ and $|\theta,\psi|$ is a permutation associated with $\theta,\psi$ defined in Section 5.5.
\end{thm}
\begin{rem}
	We will prove in Section 5.6 that $\sum_{ij}k_is_j\eta_i\vee\nu_j$ in above expression is a well defined element of $\text{BCp}(\theta\vee \psi)$ for independent $\theta,\psi$.
\end{rem}

The standard action $\mathbb{Z}_2^m \curvearrowright \mathbb{R}^m$ is real part of standard $\mathbb{Z}_2^m \curvearrowright \mathbb{C}^m$. We have a similar result up to an error term, see Theorem \ref{thm real}.

\vskip 1em
This paper is organized as follows.    In Section 2, we review some concepts of homological algebra of poset. In Section 3, we introduce a spectral sequence associated with Grothendieck fibration of poset. In Section 4, we introduce the main technique of this paper which we call it "cellular methods". In Section 5, we apply the toolkit have developed and give a presentation of cohomology ring of more general "chromatic configuration space" of standard action, the Theorem \ref{main} will be a direct corollary. In Appendix, we give a useful little trick about cellular form and an algorithm to construct cellular form.

\section{Preliminaries}
In this section, we review some known definitions and results, give some notations used frequently in this paper.

We always regard a poset $P$ as a category, i.e., there is an arrow $x\to y$ iff $x \leq y$. 

For $x\in P$, define $P_{\leq x}= \{z\in P|z\leq x\}$, $P_{<x},P_{\geq x},P_{> x}$ are defined similarly.

For $x\leq y$ in $P$, let $(x,y)$ denote the open interval $\{z\in P| x< z< y\}$ and $[x,y]$ denote the closed interval $\{z\in P| x\leq z \leq y\}$. Half open intervals $(x, y]$ and $[x, y)$ are defined similarly.

$\varDelta P$ denote the order complex of $P$, a simplicial complex whose faces are chains of $P$. $\varDelta\hspace{-0.5em}\varDelta[x,y]$ denote a pair $(\varDelta[x,y],\varDelta(x,y]\cup\varDelta[x,y))$ as in \cite{DeLongueville2001}.

A poset P is said to be a join semilattice if every pair of elements $x, y \in P $ has a join $x\vee y$, the least upper bound of $x,y$.

\subsection{Presheaf and copresheaf}
\begin{defn}
	A presheaf (of abelian group) on poset $P$ is a functor $\mathcal{F}$ from $P^{op}$ to $Ab$. For $x\leq y \in P$ and $a\in \mathcal{F}(y)$, the restriction $\mathcal{F}(x\leq y)(a)$ is abbreviated as $a|_x$ for convenient . Let $\text{PSh}(P)$ be the category of all presheaf on $P$ whose morphism is natural transformation. 
	
	Similarly, a copresheaf on poset $P$ is a functor $\mathcal{G}$ from $P$ to $Ab$, the extension $\mathcal{G}(x\leq y)(a)$ is abbreviated as $a|^y$ for $a\in \mathcal{F}(x)$, denote $\text{CoPSh}(P)$ the category of all copresheaf on $P$ whose morphism is  natural transformation.
\end{defn}

\begin{example}
	Let $A$ be an abelian group, denote $\underline{A}$ the constant presheaf with value $A$, i.e., a presheaf assigns each element in $P$ the value $A$, and all of whose restriction maps are identity. Similarly, let $\overline{A}$ be the constant copresheaf with value $A$.
\end{example}

\begin{example}
	Let $Q$ be a subset of $P$ satisfying $x,y \in Q, x\leq y \Rightarrow [x,y] \subseteq Q$, denote $\delta_Q A$ the presheaf on $P$ that has constant value $A$ on $Q$ and zero otherwise. Similarly, let $\delta^Q A$ be the copresheaf with constant value $A$ on $Q$ and zero otherwise. If $Q = \{x\}$ has only one element, we abbreviate $\delta_{\{x\}} A$ ($\delta^{\{x\}} A$) as $\delta_{x} A$ ($\delta^{x} A$).
\end{example}

\begin{defn}[Cartesian product of presheaves or copresheaves]
	Let $\mathcal{F},\mathcal{E}$ be presheaves  on $P,Q$,
	define $\mathcal{F}\times\mathcal{E}$ be the presheaf on $P\times Q$ that $(\mathcal{F}\times\mathcal{E})(x,y)=\mathcal{F}(x)\otimes\mathcal{E}(y)$ and $(a\otimes b)|_{(x,y)}=a|_x \otimes b|_y$.	The definition of cartesian product of copresheaves are similar.
\end{defn}

\begin{defn}\label{fmap}
	Let $f : P \to Q$ be a morphism of poset, i.e., preserve partial order, we denote $$f^{*} : \text{PSh}(Q) \to \text{PSh}(P)$$ the functor that associates $\mathcal{E}$ on $Q$ the presheaf $f^{*}\mathcal{E} = \mathcal{E} \circ f$. 
\end{defn}
There is a well known functor $f_*$ left adjoint to $f^*$, the following theorem is a simple version of \cite[Theorem 1.3.1]{Artin1962}
\begin{thm}
	Define $$f_* : \text{PSh}(P) \to \text{PSh}(Q)$$ the functor that associates $\mathcal{F}$ on $P$ the presheaf $f_*\mathcal{F}(y) = \text{colim}_{\{x\in P|f(x)\geq y\}^{op}} \mathcal{F}$, 
	then $f_*$ is left adjoint to the functor $f^{*}$ , i.e.,
	$$\hom_{\text{PSh}(P)}(\mathcal{F},f^{*}\mathcal{E})=\hom_{\text{PSh}(Q)}(f_*\mathcal{F},\mathcal{E})$$
	holds bifunctorially.
\end{thm}

\begin{example}
	Let $\bullet$ be the poset with single element, $j_x : \bullet \to P$ be the map that $j_x(\bullet)=x$ where $x$ is a given element of $P$, its easy to check that $j_{x*}\underline{\mathbb{Z}}$ equals the presheaf $\delta_{P_{\leq x}}\mathbb{Z}$.
\end{example}

The next lemma is dual version of \cite[Exercise 2.3.7]{Weibel1994}
\begin{lem}\label{lem of resolution}
	$j_{x*}\underline{\mathbb{Z}}$ is projective object in $\text{PSh}(P)$. For a given presheaf $\mathcal{F}$, we can always choose a projective resolution $\mathcal{F}_i\to\mathcal{F}$ such that each $\mathcal{F}_i$ is direct sum of some presheaves in the form of $j_{x*}\underline{\mathbb{Z}}$.
\end{lem}
\vskip 1cm

\subsection{Homological algebra of poset}~\\

In this subsection, $P$ is a given poset and there is a presheaf $\mathcal{F}$ and  a copresheaf $\mathcal{G}$ on $P$.
\begin{defn}[Tensor product of functor]
	The tensor product $\mathcal{F} \otimes_P \mathcal{G}$ is the coequalizer $$\mathcal{F} \otimes_P \mathcal{G} = \text{coeq}(\bigoplus_{y\leq x \in P}\xymatrixcolsep{5pc}\xymatrix{\mathcal{F}(x)\otimes \mathcal{G}(y) \ar@<0.3ex>[r]^-{\mathcal{F}(y\leq x)\otimes 1} \ar@<-0.3ex>[r]_-{1\otimes \mathcal{G}(y\leq x)} & \bigoplus_{p\in P} \mathcal{F}(p)\otimes \mathcal{G}(p)}) $$
	
\end{defn}
\begin{example}
	If $\mathcal{G}$ is a constant copresheaf $\overline{\mathbb{Z}}$, then the functor $- \otimes_P \mathcal{G} $ is natural isomorphic with $\text{colim}_{P^{op}}$. 
\end{example}
\begin{example}\label{ex yoneda}
	$j_{x*}\underline{\mathbb{Z}} \otimes_P \mathcal{G} \cong \mathcal{G}(x)$ since every element in $j_{x*}\underline{\mathbb{Z}} \otimes_P \mathcal{G}$ can be represented by an element $1\otimes a$ for some $a\in \mathcal{G}(x)$. Under this isomorphism, for any $x\leq y$, the map $j_{x*}\underline{\mathbb{Z}}\otimes_P \mathcal{G} \to j_{y*}\underline{\mathbb{Z}}\otimes_P \mathcal{G}$ induced by inclusion $j_{x*}\underline{\mathbb{Z}} \hookrightarrow j_{y*}\underline{\mathbb{Z}}$ coincide with extension map $\mathcal{G}(x) \to \mathcal{G}(y)$, because $1\otimes a$ and $1\otimes a|^y$ represent same element in $j_{y*}\underline{\mathbb{Z}} \otimes_P \mathcal{G}$.
\end{example}
The following lemma about tensor will be used in Section 3.
\begin{lem}
Let $f:P\to Q$ be a morphism of poset, $\mathcal{F}$ be a presheaf on $P$, $\mathcal{G}$ be a copresheaf on $Q$, then there is a natural isomorphism
\begin{equation}\label{eq fib tensor}
	\mathcal{F}\otimes_P f^{*}\mathcal{G}\cong f_{*}\mathcal{F}\otimes_Q\mathcal{G}
\end{equation}	
\end{lem}
\begin{proof}
	We firstly check the case of $\mathcal{F}=j_{x*}\underline{\mathbb{Z}}$. The discussion in Example \ref{ex yoneda} shows that $j_{x*}\underline{\mathbb{Z}}\otimes_P f^{*}\mathcal{G} = \mathcal{G}(f(x))$ and $f_{*}j_{x*}\underline{\mathbb{Z}}\otimes_Q\mathcal{G} = j_{f(x)*}\underline{\mathbb{Z}}\otimes_Q\mathcal{G}=\mathcal{G}(f(x))$, and these isomorphism are natural. Then the Equation \ref{eq fib tensor} holds for any $\mathcal{F}=j_{x*}\underline{\mathbb{Z}}$ and direct sum of them. For general $\mathcal{F}$, choose a resolution $\mathcal{F}_i\to \mathcal{F}$ in Lemma \ref{lem of resolution}, then these natural isomorphism $\mathcal{F}_i\otimes_S\pi^{*}\mathcal{G}\cong \pi_{*}\mathcal{F}_i\otimes_B\mathcal{G}$ yields a natural isomorphism $\mathcal{F}\otimes_S\pi^{*}\mathcal{G}\cong \pi_{*}\mathcal{F}\otimes_B\mathcal{G}$ since $\mathcal{F} \mapsto\mathcal{F}\otimes_P f^{*}\mathcal{G}$ and $\mathcal{F} \mapsto f_{*}\mathcal{F}\otimes_Q\mathcal{G}$ are both right exact.
\end{proof}

\begin{defn}
	Define $$\text{Tor}^P_n(\mathcal{F},\mathcal{G})= L_n(- \otimes_P \mathcal{G})(\mathcal{F})$$
	where $L_n(- \otimes_P \mathcal{G})$ is the $n$-th left derived functor of $- \otimes_P \mathcal{G}$.
\end{defn}

\textbf{Notation.} Write $H_n(P;\mathcal{F})=\text{Tor}^P_n(\mathcal{F},\overline{\mathbb{Z}})$ , i.e., the \textbf{homology} of poset $P$ with coefficient $\mathcal{F}$. It is the $n$-th left derived functor of $\text{colim}_{P^{op}}$.\\

\vskip .5em 
\subsection{Computation of Tor} ~
\vskip .5em 
 From now on, we make an additional restriction that \textbf{all copresheaf appears in this paper is free}  (we call $\mathcal{G}$ is free iff $\mathcal{G}(x)$ is free for any $x$).
The following lemma is a well known method to compute $\text{Tor}^P_n$.

\begin{lem}\label{compute homology}
	With the same assumption in last subsection, let $K_*(P,\mathcal{G};\mathcal{F})$ be a chain $$... \to \bigoplus_{p_0<p_1<p_2}\mathcal{G}(p_0)\otimes\mathcal{F}(p_2) \to \bigoplus_{p_0<p_1}\mathcal{G}(p_0)\otimes\mathcal{F}(p_1) \to \bigoplus_{p_0}\mathcal{G}(p_0)\otimes\mathcal{F}(p_0)$$
	The differential $\partial$ is given by example $$\partial(p_0<p_1<p_2,b\otimes c) = (p_1<p_2,b|^{p_1}\otimes c) - (p_0<p_2,b\otimes c) + (p_0<p_1,b\otimes  c|_{p_1})$$ and similarly in other degrees, $p_i\in P$, then $$\text{Tor}^P_n(\mathcal{F},\mathcal{G}) \cong H_n(K_*(P,\mathcal{G};\mathcal{F}))$$		
\end{lem}
\begin{proof}
	Recall that any copresheaf in this paper is assumed to be free, then
	a short exact sequence $0\to\mathcal{F}'\to \mathcal{F}\to \mathcal{F}''\to 0$ induce a short exact sequence of chain $0\to K_*(P,\mathcal{G};\mathcal{F}')\to K_*(P,\mathcal{G};\mathcal{F})\to K_*(P,\mathcal{G};\mathcal{F}'') \to 0$, then
	$\mathcal{F} \mapsto H_n(K_*(P,\mathcal{G};\mathcal{F}))$ is a $\delta$-functor between $\text{PSh}(P)$ and $Ab$.
	To prove it is universal, consider the functor $\mathcal{G} \mapsto H_n(K_*(P,\mathcal{G};j_{x*}\underline{\mathbb{Z}}))$. It's universal $\delta$-functor since $H_*(K_*(P,j_{y*}\overline{\mathbb{Z}};j_{x*}\underline{\mathbb{Z}}))$ equals homology of contractible $\varDelta[y,x]$ for any $y\leq x$, and $0$ otherwise . Then $H_n(K_*(P,\mathcal{G};j_{x*}\underline{\mathbb{Z}}))$ equals left derived functor of $\mathcal{G} \mapsto H_0(K_*(P,\mathcal{G};j_{x*}\underline{\mathbb{Z}})) = \mathcal{G}(x)$, which is an exact functor, so $H_n(K_*(P,\mathcal{G};j_{x*}\underline{\mathbb{Z}}))$ is trivial for positive $n$ and $\mathcal{F} \mapsto H_n(K_*(P,\mathcal{G};\mathcal{F}))$ is a universal $\delta$-functor. Notice that $$H_0(K_*(P,\mathcal{G};\mathcal{F})) \cong \mathcal{F} \otimes_P \mathcal{G}$$ so $H_n(K_*(P,\mathcal{G};\mathcal{F}))$ is natural isomorphic with $\text{Tor}^P_n(\mathcal{F},\mathcal{G})$.
	
\end{proof}


\begin{cor}
	$$H_n(\varDelta P)=\text{Tor}^P_n(\underline{\mathbb{Z}},\overline{\mathbb{Z}})=H_n(P;\underline{\mathbb{Z}})$$  
	$$H_n(\varDelta\hspace{-0.5em}\varDelta[x,y])=\text{Tor}^P_n(\delta_y\mathbb{Z},\delta^x\mathbb{Z})$$
\end{cor}

\subsection{$f$-homomorphism}~\\

In this subsection, let $f: P\to Q$ be morphism of poset. $\mathcal{F},\mathcal{E}$  be presheaves on $P,Q$, $\mathcal{G},\mathcal{H}$  be copresheaves on $P,Q$.
\begin{defn}
	$f$-homomorphism $k : \mathcal{F} \rightsquigarrow \mathcal{E}$ is a collection of homomorphisms $k_x : \mathcal{F}(x) \to \mathcal{E}(f(x))$, $x\in P$, compatible with restrictions, we denote $f\text{-Hom}(\mathcal{F},\mathcal{E})$ the group of all $f$-homomorphism from $\mathcal{F}$ to $\mathcal{E}$, there are obvious natural isomorphisms $f\text{-Hom}(\mathcal{F},\mathcal{E})\cong \text{Hom}_{\text{PSh}(P)}(\mathcal{F},f^{*}\mathcal{E})\cong\text{Hom}_{\text{PSh}(Q)}(f_*\mathcal{F},\mathcal{E})$.
	
	The definition of $f$-homomorphism of copresheaf is similar.
\end{defn}

\begin{example}[canonical $f$-homomorphism]\label{can-f}
    The canonical $f$-homomorphism $ f^*\mathcal{E} \rightsquigarrow \mathcal{E}$ is given by identity map of $\mathcal{E}(f(x))$ for any $x\in P$. In this paper, we conventionally let any arrow looks like $ f^*\mathcal{E} \rightsquigarrow \mathcal{E}$ to denote canonical $f$-homomorphism. It is easy to see that any $f$-homomorphism $k : \mathcal{F} \rightsquigarrow \mathcal{E}$ admits a unique factorization $k : \mathcal{F} \xrightarrow{j}  f^*\mathcal{E} \rightsquigarrow \mathcal{E}$ where $j$ is a morphism in $\text{PSh}(P)$. The definition of canonical $f$-homomorphism of copresheaves is similar.
\end{example}

\begin{example}[special $f$-homomorphism $\star$]\label{ex of star}
	In this example, let $y \in Q$ and $x \in P$ be a minimal element in $f^{-1}y$, then we define a $f$-homomorphism $\star : \delta_x\mathbb{Z} \rightsquigarrow \delta_y\mathbb{Z} $ to be identity map of $\mathbb{Z}$ on $x$ and be $0$ otherwise. It's easy to check $\star$ is a well defined $f$-homomorphism since it is compatible with restriction map (notice that $\star$ is not compatible with restriction map if $x$ is not a minimal element of $f^{-1}y$). If $x$ is a maximal element of $f^{-1}y$, we can define a similar $f$-homomorphism of copresheaf $\star : \delta^x\mathbb{Z} \rightsquigarrow \delta^y\mathbb{Z} $. These $f$-homomorphisms $\star$ will be used later.
	
\end{example}

\begin{example}[pull back along a diagram]\label{ex of pull back by diagram}
	Let $$\xymatrix{
		P' \ar[r]^{f'}\ar[d]^{\pi} & Q' \ar[d]^{\pi'}\\
		P \ar[r]^f & Q
	}$$
be a commutative diagram of poset, $k : \mathcal{F} \rightsquigarrow \mathcal{E}$ be a $f$-homomorphism of presheaves, define a $f'$-homomorphism $\pi^*k : \pi^*\mathcal{F} \rightsquigarrow \pi'^*\mathcal{E}$ by $(\pi^*k)_x = k_{\pi(x)}$ for any $x\in P'$. The pull back of $f$-homomorphism of copresheaves is similar.
\end{example}

\begin{thm}\label{f-hom}
	Any $f$-homomorphism of presheaf $k : \mathcal{F} \rightsquigarrow \mathcal{E}$ and $f$-homomorphism of copresheaf $t: \mathcal{G} \rightsquigarrow \mathcal{H}$ induce natural maps $\text{Tor}^f_n(k,t): \text{Tor}^P_n(\mathcal{F},\mathcal{G}) \to \text{Tor}^Q_n(\mathcal{E},\mathcal{H})$ uniquely determined by following property
	\begin{itemize}
		\item $\text{Tor}^f_0(k,t)$ equals the map $\mathcal{F} \otimes_P \mathcal{G} \to \mathcal{E} \otimes_Q \mathcal{H}$  induced by $k$ and $t$.
		\item a commutative diagram of short exact sequences $$\xymatrix{
			0\ar[r] & \mathcal{F}' \ar[r] \ar@{~>}[d]^s & \mathcal{F} \ar[r] \ar@{~>}[d] & \mathcal{F}'' \ar[r] \ar@{~>}[d]^k & 0\\
			0\ar[r] & \mathcal{E}' \ar[r]  & \mathcal{E} \ar[r]  & \mathcal{E}'' \ar[r] & 0
		}$$
		(the columns are $f$-homomorphisms) induce a commutative diagram $$\xymatrix{
			\text{Tor}^P_n(\mathcal{F}'',\mathcal{G})  \ar[r]^\delta \ar[d]^{\text{Tor}^f_n(k,t)} & \text{Tor}^P_{n-1}(\mathcal{F}',\mathcal{G}) \ar[d]^{\text{Tor}^f_{n-1}(s,t)}\\
			\text{Tor}^Q_n(\mathcal{E}'',\mathcal{H}) \ar[r]^\delta & \text{Tor}^Q_{n-1}(\mathcal{E}',\mathcal{H})
		}$$
	\end{itemize}
\end{thm}
\begin{proof}
	For the existence of $\text{Tor}^f_n(k,t)$, let $(k,t)_\#: K_*(P,\mathcal{G};\mathcal{F}) \to K_*(S,\mathcal{H};\mathcal{E})$ be the map given by example $(k,t)_\#(p_0<p_1<p_2,b\otimes c)=(f(p_0)< f(p_1)<f(p_2),t(b)\otimes k(c))$ for distinct $f(p_i)$ and zero otherwise. Let $\text{Tor}^f_*(k,t)$ be the  map of homology induced by $(k,t)_\#$, it's easy to check this map $\text{Tor}^f_*(k,t)$ satisfying the required property . 
	
	The uniqueness of $\text{Tor}^f_*(k,t)$ is proven by a standard "dimension shifting" process. Let the rows in following diagram $$\xymatrix{
		...\ar[r] & \mathcal{F}_2 \ar[r]^{d_2} \ar@{~>}[d]^{\gamma^2} & \mathcal{F}_1 \ar[r]^{d_1} \ar@{~>}[d]^{\gamma^1} & \mathcal{F}_0 \ar[r]^{d_0} \ar@{~>}[d]^{\gamma^0} & \mathcal{F}_{-1}= \mathcal{F} \ar@{~>}[d]^{k}\\
		...\ar[r] & \mathcal{E}_2 \ar[r]^{d_2} & \mathcal{E}_1 \ar[r]^{d_1} & \mathcal{E}_0 \ar[r]^{d_0}  & \mathcal{E}_{-1}= \mathcal{E}
	}$$
	be projective resolution of $\mathcal{F},\mathcal{E}$ respectively, and the map in every column is the $f$-homomorphism extending $k$ (we can always do this by "push" $\mathcal{F}_i$ to $\text{PSh}(Q)$, i.e., $f_*(\mathcal{F}_i)$ is resolution of $f_*(\mathcal{F})$, then there exist morphisms in $\text{Hom}(f_*\mathcal{F}_i,\mathcal{E}_i)$ extending $k$ by \cite[Theorem 2.2.6]{Weibel1994}) . Let $M_{i+1} = \ker(\mathcal{F}_i\xrightarrow{d_i} \mathcal{F}_{i-1})$ and $M'_{i+1} = \ker(\mathcal{E}_i\xrightarrow{d_i} \mathcal{E}_{i-1})$ for $i\geq 0$ and $M_{0}=\mathcal{F},M'_{0}=\mathcal{E}$, we have diagram of short exact sequence $$\xymatrix{
		0\ar[r] & M_{i+1} \ar[r] \ar@{~>}[d]^{\gamma^{i}} & \mathcal{F}_i \ar[r] \ar@{~>}[d]^{\gamma^i} & M_{i} \ar[r] \ar@{~>}[d]^{\gamma^{i-1}} & 0 \\
		0\ar[r] & M'_{i+1} \ar[r]  & \mathcal{E}_i \ar[r]  & M'_{i} \ar[r]  & 0
	}$$ induce following commutative diagram $$\xymatrix{
	\text{Tor}^P_n(\mathcal{F},\mathcal{G})\ar[r]_-\cong^-\delta \ar[d]^{\text{Tor}^f_n(k,t)} & \text{Tor}^P_n(M_{1},\mathcal{G}) \ar[r]_-{\cong}^-{\delta} \ar[d]^{\text{Tor}^f_{n-1}(\gamma^0,t)} &... \ar[r]_-{\cong}^-{\delta} & \text{Tor}^P_1(M_{n-1},\mathcal{G}) \ar@{^{(}->}[r]^-\delta \ar[d]^{\text{Tor}^f_{1}(\gamma^{n-2},t)} & M_{n} \otimes_P\mathcal{G} \ar[d]^{\text{Tor}^f_{0}(\gamma^{n-1},t)} \\
	\text{Tor}^Q_n(\mathcal{E},\mathcal{H})\ar[r]_-\cong^-\delta  & \text{Tor}^Q_n(M'_{1},,\mathcal{H}) \ar[r]_-{\cong}^-{\delta}  &... \ar[r]_-{\cong}^-{\delta} & \text{Tor}^Q_1(M'_{n-1},\mathcal{H}) \ar@{^{(}->}[r]^-\delta  & M'_{n}\otimes_Q\mathcal{H}
}$$ thus $\text{Tor}^f_n(k,t):\text{Tor}^P_n(\mathcal{F},\mathcal{G}) \to \text{Tor}^Q_n(\mathcal{E},\mathcal{H})$ is uniquely determined by $\text{Tor}^f_{0}(\gamma^{n-1},t) : M_{n} \otimes_P\mathcal{G} \to M'_{n}\otimes_Q\mathcal{H}$
\end{proof}

\textbf{Notation.} If above $k,t$ is canonical $f$-homomorphism (see Example \ref{can-f}), we always abbreviate $\text{Tor}^f_n(k,t)$ as $\text{Tor}^f_n$.

\begin{lem}\label{lem of pull back}
	Let $P,Q$ be join semilattice, $f: P \to Q$ be a \textbf{surjective} map preserve order and join operation, $\mathcal{F}$ be a presheaf on $Q$ and $\mathcal{G}$ be a copresheaf on $Q$ ,  then canonical $f$-homomorphism $f^*\mathcal{F} \rightsquigarrow \mathcal{F}$ and $f^*\mathcal{G} \rightsquigarrow \mathcal{G}$ induce isomorphism $$\text{Tor}^f_* : \text{Tor}^P_*(f^*\mathcal{F},f^*\mathcal{G}) \xrightarrow{\cong} \text{Tor}^Q_*( \mathcal{F},\mathcal{G})$$
	
\end{lem}
\begin{proof}
	
	For any $x \in Q$, $f^{-1}(x)$ is not empty and have a maximum element $t=\bigvee_{y\in f^{-1}(x)}y$ since $f$ is surjective and preserve join, then $f^*j_{x*}\underline{\mathbb{Z}} = j_{t*}\underline{\mathbb{Z}}$ is also a projective presheaf. Choose a projective resolution $\mathcal{F}_i\to \mathcal{F}$ as in Lemma \ref{lem of resolution}, then $ f^*\mathcal{F}_i\to  f^*\mathcal{F}$ is also a projective resolution of $f^*\mathcal{F}$ since $f^*$ is exact. 
	
	Canonical $f$-homomorphism $f^*\mathcal{F}_i \rightsquigarrow \mathcal{F}_i$ and $f^*\mathcal{G} \rightsquigarrow \mathcal{G}$ induce isomorphism between chain $f^*\mathcal{F}_i\otimes_P f^*\mathcal{G}$ and chain $ \mathcal{F}_i \otimes_Q \mathcal{G}$ by Example \ref{ex yoneda},  and then a isomorphism of homology of these two chain $\text{Tor}^P_*(f^*\mathcal{F},f^*\mathcal{G}) \cong \text{Tor}^Q_*(\mathcal{F},\mathcal{G})$.
	
\end{proof}

\subsection{Cross product}~\\

We review the definition of shuffles and the cross product.

\begin{defn}[Shuffles]
	Let $p,q$ be nonnegative integers. A $(p,q)$-shuffle is a strictly increasing map of posets $ \sigma: [p+q]\to [p]\times[q] $ (here $[i]:=\{x\in \mathbb{Z}|0\leq x \leq i\}$ ), let $\sigma_{-},\sigma_{+}$ be the first and second component of $\sigma$, i.e., $\sigma(i) = (\sigma_{-}(i),\sigma_{+}(i))$, $$I_{-}= \{i\in [p+q]|i>0,\sigma_{-}(i-1)<\sigma_{-}(i)\}$$ $$I_{+}= \{i\in [p+q]|i>0,\sigma_{+}(i-1)<\sigma_{-}(i)\}$$
	$$(-1)^\sigma = \prod_{i\in I_{-},j\in I_{+},i>j} (-1)$$
\end{defn}

There is a well known cross product $$C_i(\varDelta P)\otimes C_j(\varDelta Q) \xrightarrow{\times} C_{i+j}(\varDelta(P\times Q))$$ of order complex given by \begin{align*}
	&(p_0<p_1<...<p_i) \otimes (q_0<q_1<...<p_j) \\
	\mapsto&\sum_{\sigma\text{ is }(i,j)\text{-shuffle}}(-1)^\sigma ((p_{\sigma_{-}(0)},q_{\sigma_{+}(0)}) < (p_{\sigma_{-}(1)},q_{\sigma_{+}(1)})<...<(p_{\sigma_{-}(i+j)},q_{\sigma_{+}(i+j)}) )
\end{align*}

This cross product of chain induce a cross product of homology $$\text{Tor}_i^P(\underline{\mathbb{Z}},\overline{\mathbb{Z}})\otimes \text{Tor}_j^Q(\underline{\mathbb{Z}},\overline{\mathbb{Z}}) \xrightarrow{\times} \text{Tor}_{i+j}^{P\times Q}(\underline{\mathbb{Z}},\overline{\mathbb{Z}})$$
or $$\text{Tor}_i^P(\delta_x\mathbb{Z},\delta^y\mathbb{Z})\otimes \text{Tor}_j^Q(\delta_z\mathbb{Z},\delta^w\mathbb{Z}) \xrightarrow{\times} \text{Tor}_{i+j}^{P\times Q}(\delta_{(x,z)}\mathbb{Z},\delta^{(y,w)}\mathbb{Z} )$$
for any $y\leq x \in P,w\leq z\in Q$. It is the simplicial cross product of order complex or relative simplicial cross product of some pair of order complex.


\subsection{Known results about cohomology ring of complement of subspace arrangements}
 Let $\mathcal{A}$ be a complex subspace arrangements in complex linear space $M$, $\mathcal{P}$ be its intersection lattice (all possible intersection of elements in $\mathcal{A}$ with reverse order of inclusion, $M$ be its minimum element), $\mathcal{M}(\mathcal{A})$ be complement of $\mathcal{A}$, $\vee :  \mathcal{P}\times\mathcal{P} \to \mathcal{P}$ denote the join operator of $\mathcal{P}$.
 
 A pair $x,y \in \mathcal{P}$ satisfying the codimension condition if \begin{equation}
 	\text{codim}(x)+\text{codim}(y) = \text{codim}(x\vee y)
 \end{equation}
or equivalently, $x + y = M$, where $\text{codim}$ denote the codimension of \textbf{complex subspace}. In this case, check that $(x,y)$ is minimal element in $\vee^{-1}(x\vee y)$ by dimension reason, then we have a $\vee$-homomorphism of presheaves $\star : \delta_{(x,y)}\mathbb{Z} \to \delta_{x\vee y} \mathbb{Z}$ (see Example \ref{ex of star} for definition of $\star$). $\vee^{-1}M$ have only one element $(M,M)$, we have a $\vee$-homomorphism of copresheaves $\star : \delta^{(M,M)}\mathbb{Z} \to \delta^M \mathbb{Z}$.

With above notations, we rewrite de Longueville and Schultz's result as following
\begin{thm}[see \cite{DeLongueville2001}]\label{ssa}
	
	$$H^i(\mathcal{M}(\mathcal{A}))\cong \bigoplus_{x\in \mathcal{P} }\text{Tor}^{\mathcal{P}}_{2\cdot\text{codim}(x)-i}(\delta_x \mathbb{Z},\delta^{M}\mathbb{Z})$$	
	and the cup product maps $\text{Tor}^{\mathcal{P}}_{i}(\delta_x \mathbb{Z},\delta^{M}\mathbb{Z}) \otimes \text{Tor}^{\mathcal{P}}_{j}(\delta_y \mathbb{Z},\delta^{M}\mathbb{Z})$ to $\text{Tor}^{\mathcal{P}}_{i+j}(\delta_{x\vee y} \mathbb{Z},\delta^{M}\mathbb{Z})$, satisfying
	$$ a\cup b =\begin{cases}
		\text{Tor}_{i+j}^\vee(\star,\star)(a\times b)  & \text{ if } x+y=M\\
		0 & \text{ otherwise }	
	\end{cases}$$
where $\times$ is the cross product and $\text{Tor}_{i+j}^\vee(\star,\star)$ is the map induced by $\vee$-homomorphism $\star : \delta_{(x,y)}\mathbb{Z} \to \delta_{x\vee y} \mathbb{Z}$ and $\star : \delta^{(M,M)}\mathbb{Z} \to \delta^M \mathbb{Z}$.	
\end{thm}

\section{Fibration of poset and spectral sequence}
The following definition is a special case of Grothendieck fibration of category (or fibred category), also see \cite{Kecerdasan1998} for definition about fibration of category.
\begin{defn}\label{fibposet}
	Let $B$ be a poset, $\Psi$ be a functor from $B^{op}$ to category of poset. Let $S$ be disjoint union $\sqcup_{x\in B} \Psi(x)$, define a partial order on $S$ such that $$ \Psi(x)\ni a \leq b \in \Psi(y) \text{ iff } x\leq y \text{ and }  a \leq \Psi(x\leq y)(b) \text{ in fiber } \Psi(x)$$ We call the projection $\pi :S\to B$ (maps any elements in $\Psi(x)$ to $x$) be a fibration of poset.
\end{defn}

The following lemma is also a special case of a known result about fibration of category.

\begin{lem}\label{lem of fib}
	Let $\pi : S\to B$ be a fibration of poset constructed by $\Psi$, $\mathcal{F}$ is a presheaf on $S$, then $\pi_* \mathcal{F}$ is isomorphic with presheaf $$x \mapsto \text{colim}_{\Psi(x)^{op}}\mathcal{F}|_{\Psi(x)}$$
	where the restriction map $\text{colim}_{\Psi(y)^{op}}\mathcal{F}|_{\Psi(y)} \to \text{colim}_{\Psi(x)^{op}}\mathcal{F}|_{\Psi(x)}$ is induced by restrictions $\mathcal{F}(b) \to \mathcal{F}(\Psi(x\leq y)(b))$ for $x\leq y$ and $\pi(b) = y$.

\end{lem}

\begin{proof}
	Denote $\mathcal{H}(\mathcal{F})$ the presheaf $x \mapsto \text{colim}_{\Psi(x)^{op}}\mathcal{F}|_{\Psi(x)}$, we only need to check $\mathcal{H}$ is left adjoint to $\pi^*$, i.e., $$\text{Hom}_{\text{PSh}(S)}(\mathcal{F}, \pi^*\mathcal{E}) = \text{Hom}_{\text{PSh}(B)}(\mathcal{H}(\mathcal{F}), \mathcal{E})$$ holds bifunctorially. An element in left side is system of maps $f_a:\mathcal{F}(a) \to \mathcal{E}(\pi(a))$, compatible with restriction, it induces system of maps $\text{colim}_{\Psi(x)^{op}}\mathcal{F}|_{\Psi(x)} \to \mathcal{E}(x)$ by universal property of colimit. Conversely, if we have system of maps $g_x: \text{colim}_{\Psi(x)^{op}}\mathcal{F}|_{\Psi(x)} \to \mathcal{E}(x)$, let $f_a:\mathcal{F}(a) \to \mathcal{E}(\pi(a))$ be the composition of $\mathcal{F}(a) \to \text{colim}_{\Psi(x)^{op}}\mathcal{F}|_{\Psi(x)} \xrightarrow{g_x} \mathcal{E}(\pi(a))$ where $x = \pi(a)$. Then we obtain a natural one-to-one correspondence between $\text{Hom}_{\text{PSh}(P)}(\mathcal{F}, \pi^*\mathcal{G})$ and $\text{Hom}_{\text{PSh}(B)}(\mathcal{H}(\mathcal{F}), \mathcal{G})$.

\end{proof}
\begin{defn}
With the same assumption, denote $H_n(\Psi;\mathcal{F})$ the presheaf $$x \mapsto H_n(\Psi(x);\mathcal{F}|_{\Psi(x)})$$ where the  restriction maps $H_n(\Psi;\mathcal{F})(y) \to H_n(\Psi;\mathcal{F})(x)$ are induced by restrictions (in presheaf $\mathcal{F}$) $\mathcal{F}(b) \to \mathcal{F}(\Psi(x\leq y)(b))$ for $x\leq y$ and $\pi(b) = y$. We see that $H_0(\Psi;\mathcal{F}) = \pi_* \mathcal{F}$ by above lemma.
\end{defn}
\begin{lem}\label{derived pi}
	Let $\pi: S \to B$ be a fibration of poset constructed by $\Psi$, $\mathcal{F}$ is a presheaf on  $S$, there is natural isomorphism of presheaves
	$$L_n\pi_*\mathcal{F}\cong  H_n(\Psi;\mathcal{F})$$
\end{lem}
\begin{proof}
	$\mathcal{F} \mapsto H_n(\Psi;\mathcal{F})$ is obviously a $\delta$-functor. Notice 
	that $S_{\leq a} \cap \Psi(x) = \Psi(x)_{\leq \Psi(x\leq \pi(a))(a)}$ for $x\leq \pi(a)$ and be empty set for $x \nleq \pi(a)$ by definition of fibration, so $j_{a*}\underline{\mathbb{Z}} |_{\Psi(x)}$ is always a projective presheaf on $\Psi(x)$ and $H_n(\Psi;j_{a*}\underline{\mathbb{Z}})$ is trivial for positive $n$. Then $H_n(\Psi;-)$ is a universal $\delta$-functor, equals $L_n H_0(\Psi;-)=L_n\pi_*$.
\end{proof}
The next theorem is an easy corollary of Grothendieck spectral sequence.
\begin{thm}\label{ss fib}
	Let $\pi: S \to B$ be a fibration of poset constructed by $\Psi$, $\mathcal{F}$ is a presheaf on $S$, $\mathcal{G}$ be a copresheaf on $B$. There is a spectral sequence with $E^2_{p,q} = \text{Tor}_q^B(H_p(\Psi,\mathcal{F}),\mathcal{G})$ which converges to $\text{Tor}_*^S(\mathcal{F},\pi^{*}\mathcal{G})$.
\end{thm}
\begin{proof}
	Consider two functors $\pi_{*} : \text{PSh}(S) \to \text{PSh}(B)$ and $- \otimes_B\mathcal{G} : \text{PSh}(B) \to Ab$. It's well known that $\pi_{*}$ preserve projective object since it is left adjoint of exact functor $\pi^{*}$, then there is a Grothendieck spectral sequence with $E^2_{p,q}=\text{Tor}_q^B(L_p\pi_*\mathcal{F},\mathcal{G})$ converges to $L_*((\pi_*-)\otimes_B \mathcal{G})(\mathcal{F})$, combine the
	 Equation \ref{eq fib tensor} and Lemma \ref{derived pi}, we obtain the required result.
\end{proof}

Naturally, map between two fibrations should induce map of spectral sequence. We define map $f$ between two fibrations $S\to B, S'\to B'$ (constructed by $\Psi,\Psi'$) to be the following commutative diagram 
\begin{equation}\label{map of fib}
	\xymatrix{
		S \ar[r]^{f_s} \ar[d]^{\pi} & S' \ar[d]^{\pi'}\\
		B \ar[r]^{f_b} & B'
	}
\end{equation}

Let $\mathcal{F},\mathcal{F}'$ be presheaves on $S,S'$ and $k: \mathcal{F} \rightsquigarrow\mathcal{F}'$ be a $f_s$-homomorphism, then the restrictions $k|_{\Psi(x)}$ are $f_s|_{\Psi(x)}$-homomorphism for any $x\in B$, induce maps of homology $$H_n(\Psi(x);\mathcal{F}|_{\Psi(x)}) \to H_n(\Psi'(f_b(x));\mathcal{F}'|_{\Psi'(f_b(x))})$$ It's easy to check these maps being a $f_b$-homomorphism 
$$H_n(\Psi;\mathcal{F}) \rightsquigarrow H_n(\Psi';\mathcal{F}')$$ denoted by $H_n(f,k)$.

Meanwhile, let $\mathcal{G},\mathcal{G}'$ be copresheaves on $B,B'$ and $t: \mathcal{G}\rightsquigarrow\mathcal{G}'$ be a $f_b$-homomorphism, $\pi^*t : \pi^*\mathcal{G} \rightsquigarrow \pi'^*\mathcal{G}'$ be pull back of $t$ along the Diagram \ref{map of fib}.

\begin{lem}\label{map of fibration}
	With above assumptions and notations, we have a morphism of two associated spectral sequences, whose map of $E^2$ page equals the map $$\text{Tor}_q^B(H_p(\Psi;\mathcal{F}),\mathcal{G}) \to \text{Tor}_q^B(H_p(\Psi';\mathcal{F}'),\mathcal{G}')$$ induced by $H_p(f,k)$ and $t$, converges to the map of Tor $$\text{Tor}_*^S(\mathcal{F},\pi^{*}\mathcal{G}) \to \text{Tor}_*^S(\mathcal{F}',\pi'^{*}\mathcal{G}')$$ induced by $k$ and $\pi^*t$.
\end{lem}
\begin{proof}
	Choose projective resolutions $\mathcal{F}_i \to \mathcal{F}$, $\mathcal{F}'_i \to \mathcal{F}'$ and extends $k$ to these two resolutions, and apply $\pi_*$ to get $f_b$-homomorphism of chains 
	\begin{equation}\label{m1}
		\pi_*\mathcal{F}_i \to \pi'_*\mathcal{F}'_i
	\end{equation} induce a $f_b$-homomorphism between  the Cartan-Eilenberg resolutions of $\pi_*\mathcal{F}_i$ and $\pi'_*\mathcal{F}'_i$. Then this map of Cartan-Eilenberg resolution induce a map of Grothendieck spectral sequence in Theorem \ref{ss fib}. 

    Then we only need to check that the map $L_n\pi_*\mathcal{F} \to L_n\pi'_*\mathcal{F}'$ (induce by map \ref{m1})  coincide with $H_p(f,k)$ we have discussed above. This can be proved by a standard "dimension shifting" process as in the proof of Theorem \ref{f-hom}.
\end{proof}


\section{cellular methods}

\subsection{Cellular poset and cellular form}

In topology a CW-complex is built up by sequentially attaching cells on points, the homology can be computed by cellular chain complex which always simpler than simplicial chain or singular chain. In this subsection, we give the definition of (homological) cellular of a pair $(P,\mathcal{G})$ and study the associated cellular chain complex and morphisms between them, extending the idea in \cite{Everitt2015}.

\begin{defn}
	Let $P$ be a graded poset, $\mathcal{G}$ is a copresheaf on $P$, we define $(P,\mathcal{G})$ to be homological cellular if   $$\text{Tor}^P_i(\delta_x\mathbb{Z},\mathcal{G})=0 \text{ for } x\in P ,i\neq r(x)$$
\end{defn}

We will prove later that $(P,\mathcal{G})$ is homological cellular if and only if the following cellular form exists, with out calculating Tor by $K_*(P,\mathcal{G};\delta_x\mathbb{Z})$.
\vskip .5em
\textbf{Notation.} For any $P$-graded module $\varLambda=\bigoplus_{x\in P}\varLambda_x$, denote $\varLambda_{\leq x}=\bigoplus_{y\in P,y\leq x}\varLambda_y$, similarly, denote $\varLambda_{< x}=\bigoplus_{y\in P,y< x}\varLambda_y$

\begin{defn}\label{cellular form}
	We call a $P$-graded differential $\mathbb{Z}$-module $\varLambda(P,\mathcal{G})= \bigoplus_{x\in P}\varLambda(P,\mathcal{G})_x$, with differential $\partial$, be a cellular form of $(P,\mathcal{G})$ if  
	\begin{enumerate}
		\item $\partial$ maps $\varLambda(P,\mathcal{G})_{x}$ to $\bigoplus\limits_{y\in P,y\lessdot x}\varLambda(P,\mathcal{G})_{y}$ for any $x \in P$, so that $\varLambda(P,\mathcal{G})_{\leq x}$ is sub-chain complex for any $x\in P$.
		\item $\varLambda(P,\mathcal{G})_{\leq x}$ has trivial homology for positive degree.
		\item Denote $\epsilon$ the copresheaf $x \mapsto H_0(\varLambda(P,\mathcal{G})_{\leq x})$ with extension map $H_0(\varLambda(P,\mathcal{G})_{\leq x}) \to H_0(\varLambda(P,\mathcal{G})_{\leq y})$ induced by inclusion $\varLambda(P,\mathcal{G})_{\leq x}\hookrightarrow\varLambda(P,\mathcal{G})_{\leq y}$ for $x\leq y$. There is an \textbf{isomorphism of copresheaves} $  \epsilon \cong \mathcal{G}$.
	\end{enumerate}
\end{defn}
\begin{rem}\label{form rem}
	Note that for any rank zero element $x_0\in P$, the chain $\varLambda(P,\mathcal{G})_{\leq x_0}$ has only one degree zero term $ \varLambda(P,\mathcal{G})_{ x_0}$, so $\mathcal{G}(x_0)\cong H_0(\varLambda(P,\mathcal{G})_{\leq x_0})=\varLambda(P,\mathcal{G})_{ x_0}$ by definition of cellular form.
	From now on, if the cellular form exists, we don't distinguish $\varLambda(P,\mathcal{G})_{ x_0}$ and $\mathcal{G}(x_0)$ for any rank zero $x_0$. 
\end{rem}

\begin{rem}
	Notice that $K_*(P,\mathcal{G};\delta_x\mathbb{Z})$	is a "big" chain complex and can not be calculated in polynomial time. However, the existence of cellular form can be checked in polynomial time, see Appendix.
\end{rem}	

\begin{example}\label{OS}
	The OS-algebra $A^*(L)$ (regard it as a $L$-graded differential module and omit the product structure on it for now) of a geometric lattice $L$ is a cellular form of $(L,\delta^{\hat{0}}\mathbb{Z})$ by property of OS-algebra, see \cite{Yuzvinsky2001} or \cite{Dimca2009} for OS-algebra.
\end{example}

\begin{thm}
	The cellular form is unique if it exists.
\end{thm}
\begin{proof}
	This theorem is a corollary of a more general Theorem \ref{form map} , we prove it in Corollary \ref{form unique}.
\end{proof}

\vskip .5em
\subsection{Cellular chain}~
\vskip .5em
We can calculate Tor by cellular chain.
\begin{defn}[cellular chain]\label{cellular chain}
	Let $P$ be a graded poset, $\mathcal{G}$ is a free copresheaf on $P$, $\varLambda(P,\mathcal{G})$ is cellular form of $(P,\mathcal{G})$, $\mathcal{F}$ is a presheaf on $P$, the cellular chain complex $C_*(P,\mathcal{G};\mathcal{F})$ (with differential $d$) is defined by $$C_i(P,\mathcal{G};\mathcal{F})= \bigoplus_{x\in P,r(x)=i} \varLambda(P,\mathcal{G})_x \otimes \mathcal{F}(x) $$  $$ d(a \otimes c) = \sum_i a_i \otimes c|_{x_i}$$ where $a\in \varLambda(P,\mathcal{G})_x$, $c\in  \mathcal{F}(x)$, $\partial(a) = \sum_i a_i$, $a_i \in \varLambda(P,\mathcal{G})_{x_i}$ for some $x_i \lessdot x$. 
\end{defn}

It's easy to check that	$d$ in above definition is a differential.

\begin{thm}\label{cellular homology}
	With the same assumption of Definition \ref{cellular chain}, there is a natural isomorphism $$\text{Tor}^P_i(\mathcal{F},\mathcal{G}) \cong H_i(C_*(P,\mathcal{G};\mathcal{F}))$$
\end{thm}
\begin{proof}
	Recall that all copresheaf is assumed to be free in this paper,	
	then $\varLambda(P,\mathcal{G})_x$ is also free for any $x\in P$ by Lemma \ref{lem of form}, so a short exact sequence $0\to\mathcal{F}'\to \mathcal{F}\to \mathcal{F}''\to 0$ induce a short exact sequence of cellular chain $0\to C_*(P,\mathcal{G};\mathcal{F}')\to C_*(P,\mathcal{G};\mathcal{F})\to C_*(P,\mathcal{G};\mathcal{F}'') \to 0$ and $\mathcal{F} \mapsto H_i(C_*(P,\mathcal{G};\mathcal{F}))$ is a $\delta$-functor on $\text{PSh}(P)$.
	$C_*(P,\mathcal{G};j_{x*}\underline{\mathbb{Z}})$ equals the chain $\varLambda(P,\mathcal{G})_{\leq x}$, so
	  $H_i(C_*(P,\mathcal{G};j_{x*}\underline{\mathbb{Z}}))=0$ for any $i>0$ by the definition of cellular form, then this $\delta$-functor is universal. 
	  
	  Then we only need to check that $H_0(C_*(P,\mathcal{G};\mathcal{F}))$ is natural isomorphic with $\mathcal{F} \otimes_P\mathcal{G}$. We firstly check the case of $\mathcal{F} = j_{x*}\underline{\mathbb{Z}}$. In this case, the required isomorphism is given by $j_{x*}\underline{\mathbb{Z}}\otimes_P\mathcal{G}\cong \mathcal{G}(x)$ and $\mathcal{G}(x) = H_0(\varLambda(P,\mathcal{G})_{\leq x}) = H_0(C_*(P,\mathcal{G};j_{x*}\underline{\mathbb{Z}}))$. For general $\mathcal{F}$, choose a resolution $\mathcal{F}_i\to \mathcal{F}$ in Lemma \ref{lem of resolution}, then these natural isomorphism $H_0(C_*(P,\mathcal{G};\mathcal{F}_i)) \cong \mathcal{F}_i \otimes_P\mathcal{G}$ yields a natural isomorphism $H_0(C_*(P,\mathcal{G};\mathcal{F})) \cong \mathcal{F} \otimes_P\mathcal{G}$.
\end{proof}
\begin{cor}\label{form to cellular}
	If $(P,\mathcal{G})$ has a cellular form $\varLambda(P,\mathcal{G})$, then $(P,\mathcal{G})$ is homological cellular, $$\text{Tor}^P_{i}(\delta_x\mathbb{Z},\mathcal{G})=\varLambda(P,\mathcal{G})_x$$ if $i=r(x)$ and be $0$ otherwise.
\end{cor}
\begin{proof}
	$C_{r(x)}(P,\mathcal{G};\delta_x\mathbb{Z})=\varLambda(P,\mathcal{G})_x$ and be zero otherwise by definition of cellular chain, so $\text{Tor}^P_{r(x)}(\delta_x\mathbb{Z},\mathcal{G})=\varLambda(P,\mathcal{G})_x$ and be zero otherwise by Theorem \ref{cellular homology}.
\end{proof}

\subsection{Relation with canonical filtration}

\begin{defn}
	For $P$ a graded poset and $\mathcal{F}$ a presheaf on $P$, define the canonical filtration (by grading) of $\mathcal{F}$ be a sequence $F_0\mathcal{F} \hookrightarrow F_1\mathcal{F} \hookrightarrow ...$ such that $F_i\mathcal{F}$  equals $\mathcal{F}$ on $P_{\leq i}$ and zero otherwise.
\end{defn}

For any fixed copresheaf $\mathcal{G}$, above filtration induce a spectral sequence which converges to $\text{Tor}^P_{*}(\mathcal{F},\mathcal{G})$, denoted by ${}_{\mathcal{F}}E^i_{p,q}$.
Notice that $F_i\mathcal{F}/F_{i-1}\mathcal{F}=\bigoplus\limits_{x\in P,r(x)=p}\delta_x \mathcal{F}(x)$, then the $E^1$  page satisfying $${}_{\mathcal{F}}E^1_{p,q} = \bigoplus_{x\in P,r(x)=p}\text{Tor}^P_{p+q}(\delta_x \mathcal{F}(x),\mathcal{G})$$ If $(P,\mathcal{G})$ is homological cellular, ${}_{\mathcal{F}}E^1_{p,0} = \bigoplus_{x\in P,r(x)=p}\text{Tor}^P_{p}(\delta_x \mathcal{F}(x),\mathcal{G})$ and zero for other $q\neq 0$, then this spectral sequence collapse on $E^2$  and we have a natural isomorphism ${}_{\mathcal{F}}E^2_{p,0} \cong \text{Tor}^P_{p}(\mathcal{F},\mathcal{G})$. 

The canonical filtration of constant presheaf $\underline{\mathbb{Z}}$ is special, $${}_{\underline{\mathbb{Z}}}E^1_{p,0} = \bigoplus_{x\in P,r(x)=p}\text{Tor}^P_{p}(\delta_x \mathbb{Z},\mathcal{G})$$

\begin{lem}\label{ex form}
	If $(P,\mathcal{G})$ is homological cellular, then ${}_{\underline{\mathbb{Z}}}E^1_{*,0}$ (with differential $d_1$ of this page) is a cellular form of $(P,\mathcal{G})$.
\end{lem}
\begin{proof}
	We define $\varLambda(P,\mathcal{G})_x$ to be the term $\text{Tor}^P_{r(x)}(\delta_x \mathbb{Z},\mathcal{G})$ in $E^1$ page.
	
	For any $x\in P$, consider the presheaf $j_{x*}\underline{\mathbb{Z}}$ and the canonical filtration on it, abbreviate the associated spectral sequence as ${}_xE^i_{p,q}$, the inclusion $j_{x*}\underline{\mathbb{Z}} \hookrightarrow \underline{\mathbb{Z}}$ preserve the canonical filtration and so induce a morphism of spectral sequence, it's easy to check this morphism of $E^1$ page is inclusion $${}_xE^1_{p,0} = \bigoplus\limits_{y\leq x,r(y)=p}\varLambda(P,\mathcal{G})_y \hookrightarrow {}_{\underline{\mathbb{Z}}}E^1_{p,0}=\bigoplus\limits_{y \in P,r(y)=p}\varLambda(P,\mathcal{G})_y$$ Then  $\varLambda(P,\mathcal{G})_{\leq x}$ is always a sub-chain of $\varLambda(P,\mathcal{G})$. The homology $H_i(\varLambda(P,\mathcal{G})_{\leq x})$ is natural isomorphic with $\text{Tor}^P_i(j_{x*}\underline{\mathbb{Z}},\mathcal{G})$ as we discussed above, then $H_i(\varLambda(P,\mathcal{G})_{\leq x})=0$ for any positive $i$ since $j_{x*}\underline{\mathbb{Z}}$ is projective, $H_0(\varLambda(P,\mathcal{G})_{\leq x}) = j_{x*}\underline{\mathbb{Z}} \otimes_P \mathcal{G} = \mathcal{G}(x)$,  the map $H_0(\varLambda(P,\mathcal{G})_{\leq x}) \to  H_0(\varLambda(P,\mathcal{G})_{\leq y})$ induced by inclusion of $E^1$ coincide with extension map $\mathcal{G}(x)\to \mathcal{G}(y)$ by our discussion in Example \ref{ex yoneda} .
	
\end{proof}

Combine Lemma \ref{ex form} and Corollary \ref{form to cellular}, we obtain
\begin{cor}
	$(P,\mathcal{G})$ is homological cellular if and only if $(P,\mathcal{G})$ has a cellular form $\varLambda(P,\mathcal{G})$.
\end{cor}

\vskip 1em
\subsection{Morphism of cellular form}~
\vskip .5em
As we have discussed, a $f$-homomorphism induce map of homology, we will show that this map can be calculated by morphism of cellular form if some good condition holds.

Let $f: P \to Q $ be order preserving map of graded poset and do not increase the grading, i.e.,
\begin{equation}\label{cond dim}
r(f(x))\leq r(x)
\end{equation}
$t:\mathcal{G} \rightsquigarrow \mathcal{H}$ be a $f$-homomorphism of copresheaf, $(P,\mathcal{G})$ and $(Q,\mathcal{H})$ have cellular form $\varLambda(P,\mathcal{G}),\varLambda(Q,\mathcal{H})$ respectively.
\begin{defn}[morphism of cellular form]\label{map form}
	With above assumption, we define a map $\varPhi : \varLambda(P,\mathcal{G}) \to \varLambda(Q,\mathcal{H})$ to be a morphism of cellular form , associated with $f,t$,  if it satisfying 
	\begin{enumerate}
		\item $\varPhi$ maps $\varLambda(P,\mathcal{G})_x$ to $\varLambda(Q,\mathcal{H})_{f(x)}$, be commutative with $\partial$ if $r(f(x))=r(x)$.
		\item $\varPhi$ maps $\varLambda(P,\mathcal{G})_x$ to $0$ if $r(f(x))<r(x)$.
		\item $\varPhi$ equals $t$ on $\varLambda(P,\mathcal{G})_{x_0}$ for $r(x_0)=0$.	
	\end{enumerate}
\end{defn}

\begin{thm}\label{form map}
   With the same assumption, there exists a unique morphism $\varPhi$ of cellular form associated with $f,t$ and $\varPhi$ is commutative with differential on any $\varLambda(P,\mathcal{G})_x$ whatever $r(f(x))=r(x)$ or not.
\end{thm}
\begin{proof}
	
	We prove it by induction on rank. In the case of rank zero, this function is already uniquely defined by property (3).
	
	For $r(x)=1$, $r(f(x))=1$, $\partial$ maps $\varLambda(P,\mathcal{G})_{x}$ isomorphically on $\ker(\bigoplus_{y<x}\mathcal{G}(y) \to \mathcal{G}(x))$ by Lemma \ref{lem of form} and it is similar for $\partial$ on $\varLambda(Q,\mathcal{H})_{f(x)}$, $\Phi=t$ (on rank zero element) maps $\ker(\bigoplus_{y<x}\mathcal{G}(y) \to \mathcal{G}(x))$ to $\ker(\bigoplus_{y'<f(x)}\mathcal{H}(y') \to \mathcal{H}(f(x)))$ since $t$ is compatible with extension map of copresheaf. Then $\varPhi$ is uniquely defined by $\varPhi=\partial^{-1}\varPhi\partial$ on $\varLambda(P,\mathcal{G})_{x}$.
	
	For $r(x)=1$, $r(f(x))=0$, $a\in \varLambda(P,\mathcal{G})_{x}$, we need to check $\varPhi\partial(a)=0$ in this case. Notice that $f(y) = f(x)$ for any $y<x$ since $f$ preserving order, then we have following diagram
	$$\xymatrix{
		\mathcal{G}_x \ar[dr]^t\\
		\bigoplus\limits_{y<x,r(y)=0}\mathcal{G}_y \ar[u] \ar[r]^-{\varPhi=t}& \mathcal{H}_{f(x)}
	}$$		
	$\partial(a) \in \ker(\bigoplus\limits_{y<x,r(y)=0}\mathcal{G}_y \to \mathcal{G}_x)$ by Lemma \ref{lem of form}, then $\varPhi\partial(a)=0$ by the diagram.
	
	For $r(x)>1$, $a\in \varLambda(P,\mathcal{G})_{x}$, assume that $\varPhi$ has already uniquely defined for elements with lower rank and be commutative with $\partial$. If $r(f(x))=r(x)$, we have commutative diagram 
	$$\xymatrix{
		\varLambda(P,\mathcal{G})_x \ar[d]^{\partial}& \varLambda(Q,\mathcal{H})_{f(x)}\ar[d]^{\partial}\\
		\bigoplus_{y\lessdot x} \varLambda(P,\mathcal{G})_{y} \ar[r]^{\Phi} \ar[d]^{\partial} &\bigoplus_{y'\lessdot f(x)}\varLambda(Q,\mathcal{H})_{y'}\ar[d]^{\partial}\\
		\bigoplus_{y<x,r(y)=r(x)-2} \varLambda(P,\mathcal{G})_{y} \ar[r]^-{\Phi} & \bigoplus_{y'< f(x),r(y')=r(x)-2}\varLambda(Q,\mathcal{H})_{y'}
	}$$ by induction hypothesis, every column is exact by definition of cellular form, so $\varPhi$ is uniquely defined by $\varPhi=\partial^{-1}\varPhi\partial$ on term $\varLambda(P,\mathcal{G})_x$. We also need to check that $\varPhi\partial(a) =0$ for $r(f(x))< r(x)$.  If $r(f(x))< r(x)-1$, then $r(f(y))<r(y)$ for any $y\lessdot x$ and $\varPhi\partial(a)=0$ by property (2) of Definition \ref{map form}. If $r(f(x))=r(x)-1$, $\varPhi\partial(a) \in \varLambda(Q,\mathcal{H})_{f(x)}$ by property (1) (2) in Definition \ref{map form}, $\partial\varPhi\partial(a) =\varPhi\partial\partial(a)=0$ by induction hypothesis, $\partial$ on  $\varLambda(Q,\mathcal{H})_{f(x)}$ is injective by definition of cellular form , so $\varPhi\partial(a) =0$. This completes the proof.
	
\end{proof}

\begin{rem}
	Any morphism $\varPhi$ of cellular form is  commutative with $\partial$ by above theorem, although the commutativity is only required for the case of $r(x)=r(f(x))$ in Definition \ref{map form}. This is a useful trick when we construct a morphism of cellular form, we only need to check the commutativity on elements satisfying $r(x)=r(f(x))$. 
\end{rem}

\begin{cor}\label{form unique}
	The cellular form is unique if it exists.
\end{cor}
\begin{proof}
	Let $\varLambda,\varLambda'$ be two cellular form of pair $(P,\mathcal{G})$,  then there is morphism of cellular form $\varPhi : \varLambda \to \varLambda'$ and $\varPhi' : \varLambda' \to \varLambda$ associated with identity map of $P,\mathcal{G}$. $\varPhi\varPhi',\varPhi'\varPhi$ is also morphism of cellular form associated with identity, then $\varPhi\varPhi'=\id,\varPhi'\varPhi=\id$ by uniqueness of morphism of cellular form.
\end{proof}

We construct cellular chain map by the morphism of cellular form.

\begin{defn}
	With the same assumption, let $\mathcal{F},\mathcal{E}$ be presheaf on $P,Q$ and $k: \mathcal{F} \rightsquigarrow \mathcal{E}$ be a $f$-homomophism of presheaves, define $k_{cell}: C_i(P,\mathcal{G};\mathcal{F}) \to C_i(Q,\mathcal{H};\mathcal{E})$ by $$k_{cell}(a \otimes c) = \varPhi(a) \otimes k(c)$$ where $a\in \varLambda(P,\mathcal{G})_x$, $c\in  \mathcal{F}(x)$, $r(x)=i$, $\varPhi$ is the morphism of cellular form associated with $f,t$.
\end{defn}

\begin{lem}
	$k_{cell}$ in above definition is a chain map, i.e., commutative with $d$.
\end{lem}
\begin{proof}
	Assume an element $a\otimes c \in C_*(P,\mathcal{G};\mathcal{F})$ where $a \in \varLambda(P,\mathcal{G})_{x}$, $c \in \mathcal{F}(x)$. Let $\partial a= \sum_i a_i$ where $a_i \in \varLambda(P,\mathcal{G})_{x_i}$ for some $x_i\lessdot x$, $i\in [1,n]$. Assume $f$ maps these $x_i$ to $s$-distinct point $y_1,y_2,...,y_s\in Q$, it makes a partition $I=\{I_1,I_2,...I_s\}$ of $[1,n]$ such that $f(x_i)=y_k \Leftrightarrow x_i\in I_k$. We have
	\begin{align*}
		k_{cell}d(a \otimes c) &= k_{cell}\sum_i a_i \otimes c|_{x_i}\\
		&= \sum_i \varPhi(a_i) \otimes k(c|_{x_i})\\
		&= \sum_j (\sum_{i\in I_j} \varPhi(a_i)) \otimes k(c)|_{y_j}
	\end{align*}

    Now for $r(f(x))<r(x)$, $\sum_{i\in I_j} \varPhi(a_i) =0 $ since $\varPhi\partial a = \partial\varPhi a =0$, then we have the required equation.
    
    For $r(f(x))=r(x)$, assume $\partial\varPhi(a) = \sum_{j\in [1,s]} b_j$ where $b_j \in \varLambda(Q,\mathcal{H})_{y_j}$, Then $\partial\varPhi(a) = \varPhi\partial a \Rightarrow \sum_{i\in I_j} \varPhi(a_i) = b_j $ and $k_{cell}d(a \otimes c) = \sum_j b_j \otimes k(c)|_{y_j} = d\circ k_{cell}(a \otimes c)$
    
\end{proof}
\begin{thm}\label{cc map}
	The map $k_{cell}$ induce a map of homology, denoted by $k_{cell,*}$,  coincide with $\text{Tor}^f_n(k,t)$ in Lemma \ref{f-hom}. 
\end{thm}
\begin{proof}
	We only needs to check that $k_{cell,*}$ satisfying the property appears in Lemma \ref{f-hom}.
	
	 $k_{cell,0}$ coincide with map $\mathcal{F}\otimes_P\mathcal{G} \to \mathcal{E}\otimes_Q\mathcal{H}$ (induced by $t,k$) is obvious by definition of $k_{cell}$.
	 
	 A commutative diagram $$\xymatrix{
		0\ar[r] & \mathcal{F}' \ar[r] \ar@{~>}[d]^s & \mathcal{F} \ar[r] \ar@{~>}[d] & \mathcal{F}'' \ar[r] \ar@{~>}[d]^k & 0\\
		0\ar[r] & \mathcal{E}' \ar[r]  & \mathcal{E} \ar[r]  & \mathcal{E}'' \ar[r] & 0
	}$$
    induce a commutative diagram  of chain complex
    $$\xymatrix{
    	0\ar[r] & C_*(P,\mathcal{G};\mathcal{F}') \ar[r] \ar[d]^{s_{cell}} & C_*(P,\mathcal{G};\mathcal{F}) \ar[r] \ar[d] & C_*(P,\mathcal{G};\mathcal{F}'') \ar[r] \ar[d]^{k_{cell}} & 0\\
    	0\ar[r] & C_*(Q,\mathcal{H};\mathcal{E}') \ar[r]  & C_*(Q,\mathcal{H};\mathcal{E}) \ar[r]  & C_*(Q,\mathcal{H};\mathcal{E}'') \ar[r] & 0
    }$$
	and so a commutative diagram $$\xymatrix{
		H_n(C_*(P,\mathcal{G};\mathcal{F}'')) \ar[r]^\delta \ar[d]^{k_i} & H_{n-1}(C_*(P,\mathcal{G};\mathcal{F}')) \ar[d]^{s_{i-1}}\\
		H_n(C_*(Q,\mathcal{H};\mathcal{E}'')) \ar[r]^\delta &H_{n-1}(C_*(Q,\mathcal{H};\mathcal{E}'))
	}$$
    combine the isomorphism in Theorem \ref{cellular homology} and this completes the proof.
    
\end{proof}

\vskip 1em
\subsection{Cellular form of Cartesian product}~
\vskip 1em
The cross product of cellular chain is simple.
Assume $(P,\mathcal{G}),(Q,\mathcal{H})$ have cellular form $\varLambda(P,\mathcal{G}),\varLambda(Q,\mathcal{H})$.

\begin{lem}
	Assume copresheaves $\mathcal{G},\mathcal{H}$ are free, then the pair $(P\times Q,\mathcal{G}\times \mathcal{H})$ is homological cellular with cellular form 
	$$\varLambda(P\times Q,\mathcal{G}\times \mathcal{H})_{(x,y)}=\varLambda(P,\mathcal{G})_x \otimes \varLambda(Q,\mathcal{H})_y$$
	with differential $\partial(a \otimes b) = \partial(a) \otimes b + (-1)^{r(x)}a\otimes \partial(b)$ where $a\in \varLambda(P,\mathcal{G})_x,b \in  \varLambda(Q,\mathcal{H})_y$
\end{lem}
\begin{proof}
	Chain $\varLambda(P\times Q,\mathcal{G}\times \mathcal{H})_{\leq(x,y)}$ is isomorphic with the tensor product of two chains $\varLambda(P,\mathcal{G})_{\leq x}$ and $\varLambda(Q,\mathcal{H})_{\leq y}$ by our definition. Recall that all copresheaf in this paper is assumed to be free, $\varLambda(P,\mathcal{G})_{\leq x}$ and $\varLambda(Q,\mathcal{H})_{\leq y}$ are both free by Lemma \ref{lem of form} , then the K\"{u}nneth theorem give us all required property of $\varLambda(P\times Q,\mathcal{G}\times \mathcal{H})$ to be cellular form.
\end{proof}

\begin{example}\label{fm1}
	In this example, let $L$ be a geometric lattice, $\vee: L\times L \to L$ be join of $L$, $\star : \delta^{(\hat{0},\hat{0})}\mathbb{Z} \rightsquigarrow \delta^{\hat{0}}\mathbb{Z}$ is $\vee$-homomorphism of copresheaves, then Condition \ref{cond dim} holds for join $\vee$ since $r(I\vee J) \leq r(I) +r(J)$ holds for  geometric lattice. 
	
	Check that the product $A^*(L) \otimes A^*(L) \to A^*(L)$ of OS-algebra is the morphism of cellular form $\varLambda(L,\delta^{\hat{0}}\mathbb{Z})\otimes \varLambda(L,\delta^{\hat{0}}\mathbb{Z}) \to \varLambda(L,\delta^{\hat{0}}\mathbb{Z})$ associated with $\vee,\star$ by Definition \ref{map form}.
	
	Furthermore, let $k: \mathcal{F} \to \mathcal{E}$ be $\vee$-homomorphism of presheaves, then the cellular chain map $k_{cell} : C_*(L\times L, \delta^{(\hat{0},\hat{0})}\mathbb{Z}; \mathcal{F}) \to C_*(L, \delta^{\hat{0}}\mathbb{Z}; \mathcal{E})$ is given by $(x\otimes y)\otimes a \mapsto (x\cdot y)\otimes k(a)$ where $a\in \mathcal{F}(I,J)$ for some $I,J\in L$, $x \in A^*(L)_I, y \in A^*(L)_J$ and $x\cdot y \in A^*(L)_{I\vee J}$ be the product of $x,y$ in OS-algebra. 
\end{example}

\section{Associated subspace arrangement}

Apply the toolkit developed in last two sections, we prove the main result of this paper. Actually , we give a presentation of cohomology ring of more general space 
    $$F_{\mathbb{Z}_k^m}(\mathbb{C}^m,\Gamma)= \{(x_1,x_2,...,x_n)\in \mathbb{C}^m \times...\times \mathbb{C}^m | Gx_i \cap Gx_j = \emptyset \text{ for } (ij)\in E(\Gamma)\}$$
    where $\Gamma$ is a simple graph on vertex set $[n]$ and $E(\Gamma)$ is set of edge, we call it "chromatic orbit configuration space".

\textbf{ Notations.}\begin{enumerate}
	\item For any simple graph $\Gamma$ (without loop or multiple edge), a spanning subgraph is a subgraph with the same vertex set, let $L(\Gamma)$ be the bond lattice of graph $\Gamma$, i.e., all partition induced by connected component of spanning subgraph of $\Gamma$.
	\item Let $M=\underbrace{\mathbb{C}^m\times...\times \mathbb{C}^m}_n$, $\mathcal{A}=\{V_{(e,g)}\}$ be a subspace arrangement  in $M$ where $$V_{(e,g)}= \{(x_1,...,x_n)\in M| gx_i=x_j \text{ for } e=(ij)\}$$ for all $(e,g) \in E(\Gamma) \times \mathbb{Z}_k^m$. Obviously, $F_{\mathbb{Z}_k^m}( \mathbb{C}^m, \Gamma) = \mathcal{M}(\mathcal{A})$ as a topological subspace of $M$, $\mathcal{P}$ be the intersection lattice of $\mathcal{A}$.
	\item For any subset $U\in E(\Gamma) \times \mathbb{Z}_k^m$, denote $V_U$ the \textbf{intersection} $\bigcap_{(e,g)\in U}V_{(e,g)}$
\end{enumerate}

We study the intersection lattice $\mathcal{P}$ by a fibration $L_k^m(\Gamma)$ on $L(\Gamma)$ such that $\mathcal{P}$  is "almost equals" $L_k^m(\Gamma)$ by a comparison map $\sigma:L_k^m(\Gamma) \to \mathcal{P}$ and the lattice $L_k^m(\Gamma)$ is more easier to deal with.\\
\vskip .5em
\subsection{A fibration $L_k^m(\Gamma)$}~
\vskip .5em
Review that in Section 1, we have defined a poset $\mathcal{C}_I^m$ as set of all partial matrix of $\mathbb{Z}_k$-coloring, indexed by $I'\times [m]$, every entry of index $(p,t)$ is a $\mathbb{Z}_k$-coloring of block $p$ or an undefined "$?$". These $\mathcal{C}_I^m$ is related by restriction of coloring illustrated in  following definition.

\begin{defn}[Restriction of $\mathcal{C}_I^m$]~
	
	\begin{itemize}
		\item For any $\mathbb{Z}_k$-coloring $\phi$ (represented by a function $f: p \to \mathbb{Z}_k$) of a block $p$ and $q\subseteq p$, let $\phi|_q$ be the restriction of $\phi$ on $q$, i.e., represented by the restriction of $f$ on block $q$. We also write $? |_{q} = ?$.
		\item For any two partition  $I_1\leq I_2 \in L(\Gamma)$, we define a restriction map $$-|_{I_1} : \mathcal{C}_{I_2}^m  \to  \mathcal{C}_{I_1}^m$$ by $(\theta|_{I_1})_{p_1,t}=\theta_{p_2,t}|_{p_1}$ where $\theta \in \mathcal{C}_{I_2}^m$, $p_1$ is a block in $I'_1$, $p_2$ is the unique block in $I'_2$ such that $p_1\subseteq p_2$.
	\end{itemize}
	It's easy to check that $I \mapsto \mathcal{C}_{I}^m$ along with these restriction map is a contravariant functor from $L(\Gamma)$ to the category of poset.
\end{defn}

\begin{defn}
	Let $\pi: L_k^m(\Gamma) \to L(\Gamma)$ be the fibration of poset constructed by contravariant functor $I \mapsto \mathcal{C}_{I}^m$	
\end{defn}
\begin{rem}
	The set $L_k^m$ in Section 1 is a special case of $L_k^m(\Gamma)$ in which $\Gamma$ is a complete graph on $[n]$.
\end{rem}

\begin{lem}\label{lemjoin}
	$L_k^m(\Gamma)$ is a join semilattice, $\pi$ preserve join. Furthermore, if $\theta\lessdot\nu \in \mathcal{C}_{I}^m,\psi \in \mathcal{C}_{J}^m$, then $\theta\vee\psi \lessdot \nu\vee \psi $ or $\theta\vee\psi = \nu\vee \psi $.
\end{lem}
\begin{proof}
	For $\theta \in \mathcal{C}_{I}^m, \psi \in \mathcal{C}_{J}^m$, we are going to define a $(I\vee J)' \times [m]$ indexed partial matrix $\theta \vee \psi$. 
	
	For a block  $p\in (I\vee J)'$ and $t \in [m]$, let $\{q_i\}$ the set of blocks in $I'$ or in $J'$ that contained in $p$, write $x_i=\theta_{q_i,t}$ if $q_i$ is a block of $I'$ or $x_i=\psi_{q_i,t}$ if $q_i$ is a block of $J'$. If there is no $x_i =?$ and exists 
	\begin{equation}
		f : p \to \mathbb{Z}_k
	\end{equation}
	such that $f|{q_i}$ represents coloring $x_i$, then we define  $(\theta \vee \psi)_{p,t}=f$, otherwise, define $(\theta \vee \psi)_{p,t}=?$. It's easy to check the partial matrix $\theta \vee \psi$ we have defined is least upper bound of $\theta$ and $\psi$. 
	
	For the second part of this lemma, assume that $\nu$ is made by changing an entry of $\theta$ with index $(p,t)$ to "$?$", this operation only affect the entry in $\theta \vee \psi$ with index $(q,t)$ by above construction, where $q$ is the unique block of $I\vee J$ that contains $p$, so if the entry of $\theta \vee \psi$ with index $(q,t)$ is not "$?$", then $\theta\vee\psi \lessdot \nu\vee \psi $, otherwise, $\theta\vee\psi = \nu\vee \psi $.
\end{proof}
\vskip 1em
\subsection{Comparison of $L_k^m(\Gamma)$ and intersection lattice}~
\vskip .5em
Now, let us observe the intersection lattice $\mathcal{P}$. For example, if $n=2$ and $\Gamma$ have one edge $e$, observe that $V_{(e,g_1)}\cap V_{(e,g_2)}$ be the subspace $\{(x_1,x_2)\in M| g_{1t}x_{1t}=x_{2t} \text{ for } g_{1t}=g_{2t} \text{ and } x_{1t}=x_{2t}=0 \text{ for } g_{1t}\neq g_{2t}\}$.

Generally, define a map $\sigma: L_k^m(\Gamma) \to \mathcal{P}$ such that 
\begin{align*}
	\sigma(\theta) = \{(x_1,...,x_n)\in M &|\theta_{p,t}(i)x_{it}=\theta_{p,t}(j)x_{jt} \text{ if } \theta_{p,t}\neq ?, i,j \in p\\
	&|x_{it}=x_{jt} =0 \text{ if } \theta_{p,t}=?, i,j \in p \} 		
\end{align*}
It's easy to check $\sigma$ is a morphism of join semilattice, satisfying
\begin{equation}\label{eqcod}
	\text{codim}(\sigma(\theta)) = r_f(\theta)+m\cdot r_b(\theta)
\end{equation}
Write $r_t(\theta)=r_f(\theta)+m\cdot r_b(\theta)$.

For any partial matrix $\theta \in \mathcal{F}_{I}^m \subseteq L_k^m(\Gamma)$, we say \textbf{a row} $\theta_{p,*}$ is \textbf{undefined} if its entries are  all undefined. Note that there is two operation on $L_k^m(\Gamma)$ \textbf{do not change the image} of $\sigma$, that is (1) glue: For any undefined rows $(p_1,*),(p_2,*),...$, delete these  rows and add a new undefined row $(p_1\cup p_2\cup...,*)$. (2) split: For any undefined row $(p,*)$, delete this row and add some new undefined rows $(p_1,*),(p_2,*),...$, where $p_i\subset p$ are pairwise disjoint, $\text{Card}(p_i)>1$. 

Now, for any  $U\subseteq E(\Gamma) \times \mathbb{Z}_k^m$ and intersection $V_U \in \mathcal{P}$, we are going to define a $\theta \in L_k^m(\Gamma)$ such that $\sigma(\theta)=V_U$. All the edges appear in $U$ generate a spanning subgraph of $\Gamma$, let $I$ be the partition of connected component of this subgraph and $p\in I$ be a block.  For a given $\mathbb{Z}_k$-coloring of $p$, we say $U$ is $\phi$-compatible on index $(p,t)$ if for $i,j\in p$ and any sequence $(e_1,g_1),(e_2,g_2),...\in U$ such that $(e_1,e_2,...)$ be a path form  $i$ to $j$, $\sum_l g_{lt}$ always equals $\phi(j)^{-1}\phi(i)$. Conversely, we say that $U$ is incompatible on index $(p,t)$ if there is some different sequences connect $i,j\in p$ with different sum $\sum_l g_{lt}$. Define a partial matrix $\theta$ that $\theta_{p,t}=\phi$ if $U$ is $\phi$-compatible on index $(p,t)$ and $\theta_{p,t}=?$ if $U$ is incompatible on index $(p,t)$. Check that $V_U = \sigma(\theta)$ by definition. Then any $x\in \mathcal{P}$ have the form $\sigma(\theta)$, i.e., $\sigma$ is surjective. We define $$\alpha(x)= \bigvee_{\theta\in \sigma^{-1}x}\theta$$ be the greatest element of $\sigma^{-1}x$. Notice that $\alpha(x)$ has at most one undefined row, or otherwise we can glue these row and get a larger element in $\sigma^{-1}x$.

\begin{lem}\label{compare map0}~
	Every element in $\sigma^{-1}(x)$ is made by split undefined row of $\alpha(x)$.
\end{lem}
\begin{proof}
	Assume $\theta \in \sigma^{-1}(x)$, $\pi(\theta)=I$, $\pi(\alpha(x))=J$ and $\theta< \alpha(x)$. Then $\theta \leq \alpha(x)|_I \leq \alpha(x)$ by definition of Grothendieck fibration. $\theta<\alpha(x)|_I$ leads a contradiction $r_t(\theta)=\text{codim}(x) <r_t(\alpha(x)|_I)=\text{codim}(x)$, so $\theta=\alpha(x)|_I$ and $I<J$. Note that in this case, $\text{codim}(x)= r_f(\alpha(x))+m\cdot r(J)=r_f(\alpha(x)|_I)+m\cdot r(I)$, and so
   $r_f(\alpha(x)|_I) = r_f(\alpha(x)) + m(r(J)-r(I))$, this equation holds iff $I$ only refines some part $p\in J$ such that $\alpha(x)_{p,*}$ is undefined, so $\theta=\alpha(x)|_I$ is made by split some undefined rows of $\alpha(x)$.
\end{proof}

\vskip .5em
\subsection{Cellular form of $\mathcal{C}_{I}^m$}~
\vskip .5em
We construct the cellular form of $\mathcal{C}_{I}^m$. Recall that $\text{Ud}(\theta)$ is the set of index of undefined entry in $\theta \in \mathcal{C}_{I}^m$. We sort $\text{Ud}(\theta)$ by lexicographic order of $I\times [m]$ (the order of $I$ is chosen to be lexicographic order of $2^{[n]}$), denote $\text{Ud}_s(\theta)$ the sorted sequence.


\begin{defn}\label{diffcp}
	Define a differential $\partial : \text{Cp}(\theta)\to \bigoplus_{\psi\lessdot\theta}\text{Cp}(\psi)$ for any $\theta\in \mathcal{C}_{I}^m$ by $$\sum_{\psi\lessdot\theta}(-1)^{|\psi\lessdot\theta|}\rho_{\psi}$$ where $\rho_{\psi}$ is the  projection $\text{Cp}(\theta)\to \text{Cp}(\psi)$ and define $|\psi\lessdot\theta|=i-1$ if $\psi$ is made by filling the $i$-th undefined entry of $\theta$, i.e., the entry with index $\text{Ud}_s(\theta)_i$. (see Definition \ref{defcp} for $\text{Cp}(-)$ and $\text{BCp}(-)$)
\end{defn}

\begin{lem}
	\begin{enumerate}
		\item $\partial$ is a well defined differential on $\bigoplus_{\theta \in \mathcal{C}_{I}^m}\text{Cp}(\theta)$.
		\item $\partial$ induce a well defined differential on $\bigoplus_{\theta \in \mathcal{C}_{I}^m}\text{BCp}(\theta)$.
		\item The map $\partial : \text{Cp}(\theta)\to \bigoplus_{\psi\lessdot\theta}\text{Cp}(\psi)$ is always injective for any $\theta$ with positive rank ($r_f$).
	\end{enumerate}
	
\end{lem}
\begin{proof}~
	\begin{enumerate}
	
	\item Note that $|\psi\lessdot\psi_1|\cdot|\psi_1\lessdot\theta|=-|\psi\lessdot\psi_2|\cdot|\psi_2\lessdot\theta|$ for any two different chain $\psi\lessdot\psi_1\lessdot\theta,\psi\lessdot\psi_2\lessdot\theta$, then $\partial\partial=0$ and $\partial$ is a well defined differential on $\bigoplus_{\theta \in \mathcal{C}_{I}^m}\text{Cp}(\theta)$. 
	
	\item Let $\psi\lessdot \theta$, note that any $(l-1)$-completion of $\psi$ is also a $l$-completion of $\theta$, so projection $\text{Cp}(\theta)\to  \text{Cp}(\psi)$ maps $\text{BCp}(\theta)$ to $\text{BCp}(\psi)$, then $\bigoplus_{\theta \in \mathcal{C}_{I}^m}\text{BCp}(\theta)$ is also a $\mathcal{C}_{I}^m$-graded $\mathbb{Z}$-module with above differential.
	
	\item Assume $r_f(\theta)>0$, then for any completion $\eta$ of $\theta$, there always exists a $\psi \lessdot \theta$ such that $\eta<\psi$, so $\rho_{\psi} : \text{Cp}(\theta)\to \text{Cp}(\psi)$ maps a formal sum contains $\eta$  to a nonzero element in $\text{Cp}(\psi)$, then $\partial: \text{Cp}(\theta) \to \bigoplus_{\psi \lessdot \theta} \text{Cp}(\psi)$ is a injective map for any $\theta$ with positive rank.
\end{enumerate}
	
\end{proof}

\begin{thm}\label{cf1}
	$(\bigoplus_{\theta \in \mathcal{C}_{I}^m}\text{BCp}(\theta),\partial)$ is cellular form of pair $(\mathcal{C}_{I}^m, \overline{\mathbb{Z}})$ and  $\text{BCp}(\theta)$ is free abelian group with rank $\prod_{(p,t)\in \text{Ud}(\theta)}(k^{|p|-1}-1)$.
\end{thm}
\begin{proof}
	For $p\in I$ be a block of $I$, let $\mathcal{C}_p$ be a rank $1$ poset with  Hasse diagram  $$
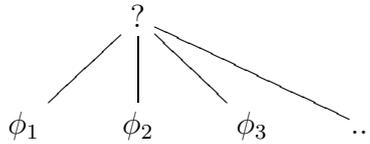
\xymatrix{
		  &? \ar@{-}[dl] \ar@{-}[d] \ar@{-}[dr] \ar@{-}[drr]\\
		 \phi_1 & \phi_2 & \phi_3 &...\\
	}$$
    where $\phi_i$ runs over all $\mathbb{Z}_k$-coloring of $p$, then $\mathcal{C}_{I}^m \cong \prod_{p\in I,t\in  [m]}\mathcal{C}_p$ as poset by definition. Check that $(\mathcal{C}_p,\overline{\mathbb{Z}})$ has cellular form $\varLambda(\mathcal{C}_p,\overline{\mathbb{Z}})_{\phi_i}=\mathbb{Z}$ and $\varLambda(\mathcal{C}_p,\overline{\mathbb{Z}})_{?}$ is group of formal sum $\sum k_i\phi_i$ satisfying $\sum k_i=0$, has rank $k^{|p|-1}-1$. Then $(\mathcal{C}_{I}^m,\overline{\mathbb{Z}})$ is also cellular and we have cellular form $$\varLambda(\mathcal{C}_{I}^m,\overline{\mathbb{Z}})_\theta= \bigotimes_{(p,t)\in I'\times [m]}\varLambda(\mathcal{C}_p,\overline{\mathbb{Z}})_{\theta_{p,t}}$$
    so $\varLambda(\mathcal{C}_{I}^m,\overline{\mathbb{Z}})_\theta$ is a free abelian group with rank $\prod_{(p,t)\in \text{Ud}(\theta)}(k^{|p|-1}-1)$.
     
    Now we apply Lemma \ref{trick of cf} to show that $(\bigoplus_{\theta \in \mathcal{C}_{I}^m}\text{BCp}(\theta),\partial)$  is also a cellular form of $(\mathcal{C}_{I}^m,\overline{\mathbb{Z}})$. The property (1) and (2) of Lemma \ref{trick of cf} holds for $(\bigoplus_{\theta \in \mathcal{C}_{I}^m}\text{BCp}(\theta),\partial)$ by definition. $\partial : \text{Cp}(\theta)\to \bigoplus_{\psi\lessdot\theta}\text{Cp}(\psi)$ is injective by above lemma,  induce a injective map $\partial : \text{BCp}(\theta)\to \bigoplus_{\psi\lessdot\theta}\text{BCp}(\psi)$ since $\text{BCp}(\theta)$ is subgroup of $\text{Cp}(\theta)$. We also need to check the exactness of $\partial$ on the position of $\bigoplus_{\psi\lessdot\theta}\text{BCp}(\psi)$, denote $\eta<_c\psi$ an element $\eta \in \text{Cp}(\psi)$ and $a= \sum_{ij} k_{ij}\eta_i <_c \psi_j \in \bigoplus_{\psi\lessdot\theta}\text{BCp}(\psi)$ for convenient, calculate $\partial(a)$ and notice that every interval in $\mathcal{C}_I^m$ is isomorphic with a free join semilattice, we have $\partial(a)=0\Rightarrow k_{ij_1}=(-1)^{|\psi_{j_1}\lessdot\theta|-|\psi_{j_2}\lessdot\theta|}k_{ij_2}$ for any $i,j_1,j_2$ and then $a = \partial(\sum_i k_{ij_0}\eta_i<_c\theta)$ where $j_0$ is the index such that $|\psi_{j_0}\lessdot\theta|=0$.
    Meanwhile, $a\in \bigoplus_{\psi\lessdot\theta}\text{BCp}(\psi) \Rightarrow \sum_{i,\eta_i<\nu}k_{ij}=0$ for any $j$ and rank $1$ element $\nu$ by definition of $\text{BCp}(-)$, so $\partial(\sum_i k_{ij_0}\eta_i<_c\theta)$ is also an element in $\text{BCp}(\theta)$. Combine above discussion, the property (3) of Lemma \ref{trick of cf} also holds and $(\bigoplus_{\theta \in \mathcal{C}_{I}^m}\text{BCp}(\theta),\partial)$ is cellular form of $(\mathcal{C}_{I}^m,\overline{\mathbb{Z}})$.
    $\text{BCp}(\theta) \cong \varLambda(\mathcal{C}_{I}^m,\overline{\mathbb{Z}})_\theta$ be a free abelian group with rank $\prod_{(p,t)\in \text{Ud}(\theta)}(k^{|p|-1}-1)$ by uniqueness of cellular form.
\end{proof}
\begin{rem}
	The presentation $(\bigoplus_{\theta \in \mathcal{C}_{I}^m}\text{BCp}(\theta),\partial)$ is more convenient to use in Section \ref{prodfib} then the tensor product $\bigotimes_{(p,t)\in I' \times [m]}\varLambda(\mathcal{C}_p,\overline{\mathbb{Z}})$ although they are the same cellular form.
\end{rem}

\subsection{Additive structure of $H^*(F_{\mathbb{Z}_k^m}(\mathbb{C}^m,\Gamma))$}

We know \begin{equation}\label{e2}H^i(F_{\mathbb{Z}_k^m}(\mathbb{C}^m,\Gamma))=H^i(\mathcal{M}(\mathcal{A}))\cong \bigoplus_{x\in \mathcal{P} }\text{Tor}^{\mathcal{P}}_{2\cdot\text{codim}(x)-i}(\delta_x \mathbb{Z},\delta^{M}\mathbb{Z})\end{equation} by Goresky and MacPherson's result. (Let $\text{codim}(x)$ to be the codimension of \textit{complex} subspace $x$)

Pull back $\delta_x \mathbb{Z},\delta^{M}\mathbb{Z}$ by  $\sigma : L_k^m(\Gamma) \to \mathcal{P}$, we have $\sigma^*\delta_x \mathbb{Z}= \delta_{\sigma^{-1}x}\mathbb{Z}$ and $\sigma^*\delta^{M}\mathbb{Z} = \delta^{\hat{0}}\mathbb{Z}$.
Apply Lemma \ref{lem of pull back}, we have 
\begin{equation}\label{e1}
	\text{Tor}^{\mathcal{P}}_*(\delta_x \mathbb{Z},\delta^{M}\mathbb{Z})= \text{Tor}^{L_k^m(\Gamma)}_*(\delta_{\sigma^{-1}x}\mathbb{Z},\delta^{\hat{0}}\mathbb{Z})
\end{equation}
where $L_k^m(\Gamma)$ is more easier to deal with by Grothendieck spectral sequence and cellular methods.

\begin{lem}\label{addstr}
	There is an isomorphism $$\text{Tor}^{L_k^m(\Gamma)}_*(\delta_{\sigma^{-1}x}\mathbb{Z},\delta^{\hat{0}}\mathbb{Z})\cong \bigoplus_{\theta \in \sigma^{-1}x} A^*(L(\Gamma))_{\pi(\theta)} \otimes \text{BCp}(\theta)$$
\end{lem}
\begin{proof}
	Apply the spectral sequence in Theorem \ref{ss fib}, let $\Psi$ be the contravariant functor $I \mapsto \mathcal{C}_I^m$ , presheaf $\mathcal{F}$ on $L_k^m(\Gamma)$ be $\delta_{\sigma^{-1}x}\mathbb{Z}$, and copresheaf $\mathcal{G}$ on $L(\Gamma)$ be $\delta^{\hat{0}}\mathbb{Z}$.  Notice that any element in $\sigma^{-1}(x)$ have the form $\alpha(x)|_I$ by Lemma \ref{compare map0}, so there is no two elements of $\sigma^{-1}x$ located in same fiber $\mathcal{C}_{I}^m$, then the restriction $\mathcal{F}|_{\Psi(I)}$ (as presheaf on $\Psi(I)$) is either zero or $\delta_{\theta}\mathbb{Z}$ for $\theta\in \sigma^{-1}x$. Combine the cellular form of $\mathcal{C}_I^m$, we have that presheaf $H_n(\Psi;\mathcal{F})$ equals $\bigoplus_{\theta}\delta_{\pi(\theta)}\text{BCp}(\theta)$ where $\theta$ runs over elements in $ \sigma^{-1}x$ such that $r_f(\theta)=n$, and equals zero if there is no such $\theta$, so $E^2_{p,q} = \text{Tor}_q^{L(\Gamma)}(H_p(\Psi;\mathcal{F}),\delta^{\hat{0}}\mathbb{Z})=\bigoplus_\theta\text{Tor}_q^{L(\Gamma)}(\delta_{\pi(\theta)}\text{BCp}(\theta),\delta^{\hat{0}}\mathbb{Z})$ for $\theta\in\sigma^{-1}x,r_f(\theta)=p$. 
	
	We have known that $(L(\Gamma),\delta^{\hat{0}}\mathbb{Z})$ is cellular and its cellular form is given by OS-algebra $A^*(L(\Gamma))$ in Example \ref{OS}, which means $\text{Tor}_q^{L(\Gamma)}(\delta_IF,\delta^{\hat{0}}\mathbb{Z})=A^*(L(\Gamma))_I\otimes F$ for $r(I)=q$ and $0$ otherwise, so $E^2_{p,q} = \bigoplus_{\theta}A^*(L(\Gamma))_{\pi(\theta)}\otimes \text{BCp}(\theta)$ where  $\theta\in\sigma^{-1}x,r_f(\theta)=p, r_b(\theta)=q$. Notice that $r_f(\theta)=r_f(\alpha(x))+ m(r_b(\alpha(x))-r_b(\theta))$ for any $\theta\in\sigma^{-1}x$, so the nonzero element of $E^2$ page only appears on index $(p,q)$ such that $p= r_f(\alpha(x))+m(r_b(\alpha(x))-q)$, then this spectral sequence collapse on $E^2$ and we have an isomorphism
	\begin{equation}\label{e3}
		\text{Tor}^{L_k^m(\Gamma)}_n(\delta_{\sigma^{-1}x}\mathbb{Z},\delta^{\hat{0}}\mathbb{Z})\cong \bigoplus_{\theta \in \sigma^{-1}x,r_b(\theta)+r_f(\theta)=n} A^*(L(\Gamma))_{\pi(\theta)} \otimes \text{BCp}(\theta)
	\end{equation} 
	
\end{proof}

Combine the Equation \ref{e2}, \ref{e1} and \ref{e3}, we obtain the cohomology group of $F_{\mathbb{Z}_k^m}(\mathbb{C}^m,\Gamma)$
\begin{thm}
	For $m>1$, $H^*(F_{\mathbb{Z}_k^m}( \mathbb{C}^m, \Gamma))$ is $L_k^m(\Gamma)$-graded,
	every piece \\$H^*(F_{\mathbb{Z}_k^m}( \mathbb{C}^m, \Gamma))_{\theta}$ is a free abelian group $A^*(L(\Gamma))_{\pi(\theta)}\otimes \text{BCp}(\theta)$ for any $\theta \in L_k^m(\Gamma)$, has rank 
	\begin{equation}\label{eq of rank1}
		|\mu(\hat{0},\pi(\theta))|\prod_{(p,t)\in \text{Ud}(\theta)} (k^{|p|-1}-1)
	\end{equation}
	where $A^*(L(\Gamma))$ is OS-algebra of bond lattice $L(\Gamma)$, $A^*(L(\Gamma))_{\pi(\theta)}$ is the piece of grading $\pi(\theta)$, $\mu(-,-)$ is M\"{o}bius function.
	Every element in $H^*(F_{\mathbb{Z}_k^m}( \mathbb{C}^m, \Gamma))_{\theta}$ has degree $(2m-1)r_b(\theta)+r_f(\theta)$. 
\end{thm}
~
\subsection{Details about join operation of $L_k^m(\Gamma)$}~
\vskip .5em
Before studying the cup product on $H^*(F_{\mathbb{Z}_k^m}(\mathbb{C}^m,\Gamma))$, we need some lemma about join operation of $L_k^m(\Gamma)$. We call $I,J$ be \textbf{independent} in a geometric lattice if $r(x)+r(y)=r(x\vee y)$. 

\begin{lem}\label{lemjoin0}
	Let $I,J$ be independent in $L(\Gamma)$, $\theta \in \mathcal{C}_I^m, \psi \in \mathcal{C}_J^m$ and $r_f(\theta)=r_f(\psi)=0$, then $\theta \vee \psi$ is also a rank zero element in $\mathcal{C}_{I\vee J}^m$.
\end{lem}
\begin{proof}
	Write $G=\mathbb{Z}_k^m$, $G_{\hat{0}}$ be the group of map $[n] \to G$ and $G_{I}$ be the subgroup $$\{f\in G_{\hat{0}}| f(i)=f(j) \text{ if } i,j \text{ contained in same block of } I\}$$ for any $I\in L(\Gamma)$. The rank zero elements in $\mathcal{C}_I^m$ is one to one correspondence with coset of $G_{I}$ in $G_{\hat{0}}$ by our definition of $\mathbb{Z}_k$-coloring. $G_I\cap G_J = G_{I\vee J}$. In the case of independent $I,J$, we have $r(x)+r(y)=r(x\vee y)$, then $G_I + G_J = G_{\hat{0}}$ by calculate the rank of each side, so any coset  of $G_{I}$ and any coset of $G_{J}$ have a nonempty intersection as a coset of $G_{I\vee J}$ in $G_{\hat{0}}$, which means, for any rank zero elements $\theta \in \mathcal{C}_I^m,\psi \in \mathcal{C}_J^m$, we have a  rank zero element $x\in \mathcal{C}_{I\vee J}^m$ such that $\theta<x,\psi<x$, it must equals the join $\theta\vee\psi$ since $r_f(x)=0$ and $\theta\vee\psi \in \mathcal{C}_{I\vee J}^m$.
\end{proof}

\begin{lem}\label{lemjoin1}
	For independent $I,J \in L(\Gamma)$ and $\theta \in \mathcal{C}_I^m, \psi \in \mathcal{C}_J^m$, we have $$r_f(\theta\vee \psi)\leq r_f(\theta)+r_f(\psi)$$
\end{lem}
\begin{proof}
	We prove it by induction on the sum $r_f(\theta)+r_f(\psi)$. For $r_f(\theta)+r_f(\psi)=0$, $r_f(\theta\vee \psi)=0$ by last lemma. For $r_f(\theta)+r_f(\psi)>0$, without lose of generality, assume $r_f(\theta)>0$, choose an element $\nu \lessdot \theta$, then $r_f(\nu \vee \psi)\leq r_f(\theta)-1+r_f(\psi)$ by induction hypothesis, $r_f(\nu \vee \psi)=r_f(\theta\vee \psi)-1$ or $r_f(\nu \vee \psi)=r_f(\theta\vee \psi)$ by Lemma \ref{lemjoin}, then $r_f(\theta\vee \psi)\leq r_f(\theta)+r_f(\psi)$.
\end{proof}

\begin{defn}
	We call twe elements $\theta \in \mathcal{C}_I^m, \psi \in \mathcal{C}_J^m$ is \textbf{independent} in $L_k^m(\Gamma)$ if $I,J$ is independent and $r_f(\theta\vee \psi)= r_f(\theta)+r_f(\psi)$.
\end{defn}

Assume $\theta\in\mathcal{C}_I^m$ and $\eta$ be a completion of $\theta$, note that the interval $[\eta,\theta]\subset \mathcal{C}_I^m$ is a free join semilattice generated by $\text{Ud}(\theta)$. 
\begin{lem}\label{lemjoin2}
	Let $\theta \in \mathcal{C}_I^m, \psi \in \mathcal{C}_J^m$ be independent in $L_k^m(\Gamma)$ and $\eta,\mu$ is completion of $\theta,\psi$, then $$[\eta,\theta]\times[\mu,\psi]\xrightarrow{\vee} [\eta\vee\mu,\theta\vee\psi]$$ is an isomorphism of free join semilattice.
\end{lem}
\begin{proof}
	For any elements $x\in [\eta,\theta],y \in [\mu,\psi]$, $x,y$ is also independent by Lemma \ref{lemjoin}, \ref{lemjoin1} and an induction. Then the map $[\eta,\theta]\times[\mu,\psi]\xrightarrow{\vee} [\eta\vee\mu,\theta\vee\psi]$ is a morphism of free join semilattice which preserve rank and the maximal element of each side have the same rank, so it must be an isomorphism.
\end{proof}

Note that the left side $[\eta,\theta]\times[\mu,\psi]$ is a free join semilattice generated by $\text{Ud}(\theta)\sqcup\text{Ud}(\psi)$ and the right side is a free join semilattice generated by $\text{Ud}(\theta\vee \psi)$, then above isomorphism give us a bijection $ \text{Ud}(\theta)\sqcup\text{Ud}(\psi) \to \text{Ud}(\theta\vee \psi)$, also denoted by $\vee$.

\begin{defn}[Permutation $|\theta,\psi|$]\label{permu}
	As we do in Section 5.3, sort $\text{Ud}(\theta)$ by lexicographic order of $I\times [m]$, denote $\text{Ud}_s(\theta)$ the sorted sequence. Let $\text{Ud}_s(\theta)\sqcup\text{Ud}_s(\psi)$ be the concatenation of $\text{Ud}_s(\theta)$ and $\text{Ud}_s(\psi)$, then the sequence $\vee(\text{Ud}_s(\theta)\sqcup\text{Ud}_s(\psi))$ is a permutation of sequence $\text{Ud}_s(\theta\vee \psi)$ for independent $\theta,\psi$ since $\vee : \text{Ud}(\theta)\sqcup\text{Ud}(\psi) \to \text{Ud}(\theta\vee \psi)$ is a bijection. Denote $|\theta,\psi|$ this permutation for independent $\theta,\psi$.
\end{defn}
\vskip .5em
\subsection{Multiplicative structure of cellular form of $\mathcal{C}_I^m$}\label{prodfib}~
\vskip .5em
In this subsection, let $I,J$ be independent elements in $L(\Gamma)$, then $$\vee : \mathcal{C}_I^m \times \mathcal{C}_J^m \to \mathcal{C}_{I\vee J}^m$$ is a morphism of graded poset that do not increasing the rank by Lemma \ref{lemjoin1}, so there exists a morphism of cellular form 
$$\Phi : \varLambda(\mathcal{C}_I^m \times \mathcal{C}_J^m, \overline{\mathbb{Z}})_{(\theta,\psi)}=\text{BCp}(\theta)\otimes\text{BCp}(\psi) \to \varLambda(\mathcal{C}_{I\vee J}^m, \overline{\mathbb{Z}})_{\theta\vee\psi}=\text{BCp}(\theta\vee\psi)$$ 
associated with join $\vee : \mathcal{C}_I^m \times \mathcal{C}_J^m \to \mathcal{C}_{I\vee J}^m$ and canonical $\vee$-homomorphism $\overline{\mathbb{Z}} \rightsquigarrow \overline{\mathbb{Z}}$. We will give an explicit formula of $\Phi$ in this subsection.

\begin{lem}
	For independent $\theta,\psi$, define a product $\times:\text{Cp}(\theta) \otimes \text{Cp}(\psi) \to \text{Cp}(\theta\vee\psi)$ given by linear extension of $\eta \otimes\nu \mapsto \text{sgn}(|\theta,\psi|) \eta \vee \nu$ where $\eta,\nu$ is completion of $\theta,\psi$ respectively (note that $r_f(\eta \vee \nu)=0$ by Lemma \ref{lemjoin0}, so this map is well defined). Then the product $\times$ is commutative with differentials in Definition \ref{diffcp}.
\end{lem}
\begin{proof}
	We denote $(\eta <_c \theta)$ the generator $\eta$ in group $\text{Cp}(\theta)$, then for any $\eta,\nu$  completes $\theta,\psi$,
	\begin{equation}\label{ep1}
		\begin{aligned}
			\times \circ\partial((\eta <_c \theta)\otimes(\nu <_c \psi))=\sum_{\theta'\lessdot\theta,\theta'\geq \eta}\text{sgn}(|\theta',\psi|)(-1)^{|\theta'\lessdot\theta|}(\eta\vee \nu <_c \theta'\vee \psi) \\
			+(-1)^{r_f(\theta)}\sum_{\psi'\lessdot\psi,\psi'\geq\nu}\text{sgn}(|\theta,\psi'|)(-1)^{|\psi'\lessdot\psi|}(\eta\vee \nu <_c \theta\vee \psi')
		\end{aligned}
	\end{equation}
	On the other side,
	\begin{equation}\label{ep2}
		\partial \circ \times ((\eta <_c \theta)\otimes(\nu <_c \psi))=\text{sgn}(|\theta,\psi|)\sum_{x\lessdot \theta\vee\psi, x\geq \eta\vee\nu} (-1)^{|x\lessdot \theta\vee\psi|}(\eta\vee\nu <_c x)
	\end{equation}
	We know there is a one to one correspondence $$\{\theta'|\theta'\lessdot\theta, \theta'\geq \eta \}\sqcup \{\psi'|\psi'\lessdot\psi,\psi'\geq\nu\} \xrightarrow{(-\vee \psi) \sqcup (\theta\vee -)} \{x|x\lessdot \theta\vee\psi, x\geq \eta\vee\nu\}$$ be Lemma \ref{lemjoin2}, so expression \ref{ep1} and \ref{ep2} equals since 
	$$\text{sgn}(|\theta',\psi|)(-1)^{|\theta'\lessdot\theta|} = \text{sgn}(|\theta,\psi|)(-1)^{|\theta'\vee \psi\lessdot \theta\vee\psi|}$$
	$$\text{sgn}(|\theta,\psi'|)(-1)^{|\psi'\lessdot\psi|} = (-1)^{r_f(\theta)}\text{sgn}(|\theta,\psi|)(-1)^{|\theta\vee \psi'\lessdot \theta\vee\psi|}$$
	The above two equations can be checked by an elementary combinatorial argumentation and we omit the details.
\end{proof}

\begin{thm}\label{mf1}
	Let $I,J$ be independent in $L(\Gamma)$, $\vee: \mathcal{C}_I^m \times \mathcal{C}_J^m \to \mathcal{C}_{I\vee J}^m$ be the join operation in $L_k^m(\Gamma)$,
	then above product $\times : \text{Cp}(\theta) \otimes \text{Cp}(\psi) \to \text{Cp}(\theta\vee\psi)$ induce a well defined map $\Phi: \text{BCp}(\theta)\otimes \text{BCp}(\psi) \to \text{BCp}(\theta\vee\psi)$ for independent $\theta,\psi$, we also define $\Phi$ to be zero on dependent $\theta,\psi$.  Then $\Phi$ is morphism of cellular form associated with $\vee$ and canonical $\overline{\mathbb{Z}}\rightsquigarrow\overline{\mathbb{Z}}$
\end{thm}
\begin{proof}
	$\Phi$ fit the Definition \ref{map form} by last lemma if it is well defined. 
	We shell prove that $\Phi$ is well defined by induction. This is trivial for $\Phi$ on rank $0$ pair $(\theta,\psi)$. For rank $1$ pair $(\theta,\psi)$, without lose of generality, assume $r_f(\theta)=1,r_f(\psi)=0$, then every element in $\text{BCp}(\theta)\otimes \text{BCp}(\psi)$ have the form $(\sum k_i\eta_i) \otimes \psi$ where $\sum k_i=0$. Note that $\sum k_i\eta_i \times \psi= \sum k_i\eta_i\vee \psi$ by definition of $\times$, this element is located in $\text{BCp}(\theta\vee\psi)$ since $r_f(\theta\vee\psi)=1$.  
	
	Now assume $r_f(\theta)+r_f(\psi)>1$ and $\Phi$ is well defined for every pair with lower rank than $(\theta,\psi)$, $a \in \text{BCp}(\theta)\otimes \text{BCp}(\psi)$. Then $\partial \circ \times (a) = \times \circ \partial (a)$ by above lemma, $\times \circ \partial (a) \in \bigoplus_{\nu\lessdot\theta\vee\psi} \text{BCp}(\nu)$ by induction hypothesis. Note that $\partial \circ \times \circ \partial (a)=0$, so $\partial \circ \times (a)$ is located in the kernel of $\partial$ on $ \bigoplus_{\nu\lessdot\theta\vee\psi} \text{BCp}(\nu)$, then $\times (a) \in \text{BCp}(\theta\vee\psi)$ by property of cellular form .
	
\end{proof}

\vskip .5em
\subsection{Multiplicative structure on $H^*(F_{\mathbb{Z}_k^m}(\mathbb{C}^m,\Gamma))$}~
\vskip .5em
Recall that, by de Longueville and Schultz's result in \cite[Theorem 5.2]{DeLongueville2001}, the cup product on $H^*(\mathcal{M}(\mathcal{A}))$ is given by the composition of cross product and 
$\text{Tor}_{*}^\vee(\star,\star) : \text{Tor}_{*}(\delta_{(x,y)}\mathbb{Z},\delta^{(M,M)}\mathbb{Z}) \to \text{Tor}_{*}(\delta_{x\vee y} \mathbb{Z},\delta^M \mathbb{Z})$ induced by $\vee$-homomorphism $\star : \delta_{(x,y)}\mathbb{Z} \to \delta_{x\vee y} \mathbb{Z}$ and $\star : \delta^{(M,M)}\mathbb{Z} \to \delta^M \mathbb{Z}$ if $x+y=M$ and be $0$ otherwise.

The cross product is an isomorphism by K\"{u}nneth theorem since $\text{Tor}_{*}(\delta_{x}\mathbb{Z},\delta^{M}\mathbb{Z})$ is always free by results in Section 5.4, so we only need to study the map $\text{Tor}_{*}^\vee(\star,\star)$.

Now, for any pair $(x,y)$ satisfying codimension condition, pull back these two $\vee$-homomorphism $\star$ along diagram
$$\xymatrix{
	L_k^m(\Gamma) \times L_k^m(\Gamma) \ar[r]^-\vee \ar[d]^{\sigma\times\sigma} & L_k^m(\Gamma) \ar[d]^\sigma\\
	\mathcal{P}\times \mathcal{P} \ar[r]^-\vee  & \mathcal{P} 
}$$	
we have following commutative diagram of presheaves and copresheaves
$$\xymatrix{
	\sigma^*\delta_{(x,y)}\mathbb{Z} \ar@{~>}[r]^-{\sigma^*\star} \ar@{~>}[d] & \sigma^{*}\delta_{x\vee y}\mathbb{Z} \ar@{~>}[d] & \delta^{(\hat{0},\hat{0})}\mathbb{Z} \ar@{~>}[r]^-{\star} \ar@{~>}[d] &\delta^{\hat{0}}\mathbb{Z} \ar@{~>}[d]\\
	\delta_{(x,y)}\mathbb{Z} \ar@{~>}[r]^-{\star} &\delta_{x\vee y}\mathbb{Z}
	& \delta^{(M,M)}\mathbb{Z} \ar@{~>}[r]^-{\star} &\delta^{M}\mathbb{Z}
}$$
where vertical arrows in above two diagrams are canonical, this two diagrams induce a diagram of Tor 
$$\xymatrixcolsep{4pc}\xymatrix{
	\text{Tor}^{L_k^m(\Gamma)\times L_k^m(\Gamma)}_*(\sigma^*\delta_{(x,y)}\mathbb{Z}, \delta^{(\hat{0},\hat{0})}\mathbb{Z}) \ar[r]^-{\text{Tor}^{\vee}_*(\star,\star)} \ar[d]^\cong_{\text{Tor}_*^{\sigma\times\sigma}} & \text{Tor}^{L_k^m(\Gamma)}_*(\sigma^{*}\delta_{x\vee y}\mathbb{Z}, \delta^{\hat{0}}\mathbb{Z}) \ar[d]^\cong_{\text{Tor}_*^{\sigma}} \\
	\text{Tor}^{\mathcal{P}\times\mathcal{P}}_*(\delta_{(x,y)}\mathbb{Z} ,\delta^{(M,M)}\mathbb{Z}) \ar[r]^-{\text{Tor}^{\vee}_*(\star,\star)} & \text{Tor}^{\mathcal{P}}_*(\delta_{x\vee y}\mathbb{Z}, \delta^{M}\mathbb{Z})
}$$
The vertical arrows are isomorphism by Lemma \ref{lem of pull back} , then we only need to study the map $ \text{Tor}^{\vee}_*(\star,\star)$ in first row. This can be calculated by the map of spectral sequence in Lemma \ref{map of fibration}, in this case, let the map of fibration to be $$\xymatrix{
	L_k^m(\Gamma)\times L_k^m(\Gamma) \ar[r]^-\vee \ar[d]^{\pi\times\pi}& L_k^m(\Gamma) \ar[d]^{\pi} \\
	L(\Gamma) \times L(\Gamma) \ar[r]^-\vee & L(\Gamma)
}$$
we also write $\Psi: (I,J)\mapsto \mathcal{C}_I^m \times \mathcal{C}_J^m$, $\Psi' : I \mapsto \mathcal{C}_I^m$, $\mathcal{F}= \sigma^*\delta_{(x,y)}\mathbb{Z}$, $\mathcal{F}'= \sigma^{*}\delta_{x\vee y}\mathbb{Z}$, $\mathcal{G}=\delta^{(\hat{0},\hat{0})}\mathbb{Z}$, $\mathcal{G}'=\delta^{\hat{0}}\mathbb{Z}$, $k=\sigma^*(\star)$. Note that $E^2$ page of each side is natural isomorphic with the resulting Tor as we discussed in the proof of Lemma \ref{addstr}, so we only need to study the map of $E^2$.

Review that in Expample \ref{fm1} , we have already known the morphism of cellular form $\varLambda(L,\delta^{\hat{0}}\mathbb{Z})\otimes \varLambda(L,\delta^{\hat{0}}\mathbb{Z}) \to \varLambda(L,\delta^{\hat{0}}\mathbb{Z})$ associated with $\vee,\star$ is given by product of OS-algebra for any geometric lattice $L$,
then the $E^2$ page in Lemma \ref{map of fibration}
$$\text{Tor}_q^B(H_p(\Psi;\mathcal{F}),\mathcal{G}) \to \text{Tor}_q^B(H_p(\Psi';\mathcal{F}'),\mathcal{G}')$$
is given by cellular chain map $(a\otimes b) \otimes c \mapsto (a\cdot b)\otimes H_p(\vee;k)(c)$ where $c\in H_p(\Psi;\mathcal{F})$, $a\in A^*(L(\Gamma))_I,b\in A^*(L(\Gamma))_J$ and $a\cdot b$ be the product in OS-algebra.

The product of OS-algebra preserve rank, $a\cdot b=0$ for dependent $I,J$, so we only need to consider the morphism $H_p(\vee;k)$ on independent $I,J$, i.e., the map of homology $$H_*(\mathcal{C}_I^m \times \mathcal{C}_J^m; \delta_{(\theta,\psi)}\mathbb{Z})\to H_*(\mathcal{C}_{I\vee J}^m; \delta_{\theta\vee \psi}\mathbb{Z})$$ induced by join $\vee : \mathcal{C}_I^m \times \mathcal{C}_J^m \to \mathcal{C}_{I\vee J}^m$ and $\vee$-homomorphism $\star : \delta_{(\theta,\psi)}\mathbb{Z} \rightsquigarrow \delta_{\theta\vee \psi}\mathbb{Z}$ for independent $I,J$ where $\pi(\theta)=I, \pi(\psi)=J, \sigma(\theta)=x, \sigma(\psi)=y$. This map is just the morphism of cellular form $$\Phi : \varLambda(\mathcal{C}_I^m \times \mathcal{C}_J^m, \overline{\mathbb{Z}})_{(\theta,\psi)}=\text{BCp}(\theta)\otimes\text{BCp}(\psi) \to \varLambda(\mathcal{C}_{I\vee J}^m, \overline{\mathbb{Z}})_{\theta\vee\psi}=\text{BCp}(\theta\vee\psi)$$
associated with join $\vee : \mathcal{C}_I^m \times \mathcal{C}_J^m \to \mathcal{C}_{I\vee J}^m$ and canonical $\vee$-homomorphism $\overline{\mathbb{Z}} \rightsquigarrow \overline{\mathbb{Z}}$ by Theorem \ref{cc map}. Recall that we assume $x,y$ satisfying codimension condition (otherwise, the cup product is zero and there is nothing to talk), so $\theta,\psi$ is also independent in $L_k^m(\Gamma)$ by Equation \ref{eqcod}. This
  morphism $\Phi$ has been given in Theorem \ref{mf1}, it is the linear extension of $$\eta \otimes\nu \mapsto \text{sgn}(|\theta,\psi|) \eta \vee \nu$$ where $\eta,\nu$ is completion of $\theta,\psi$ respectively.

Combine above discussion, we obtain the multiplicative structure on $H^*(F_{\mathbb{Z}_k^m}(\mathbb{C}^m,\Gamma))$.
\begin{thm}
	For any two elements $x=a\otimes\sum k_i\eta_i \in H^*(F_{\mathbb{Z}_k^m}( \mathbb{C}^m, \Gamma))_{\theta}, y=b\otimes\sum s_i\nu_i \in H^*(F_{\mathbb{Z}_k^m}( \mathbb{C}^m, \Gamma))_{\psi}$, the cup product $x\cup y$ is contained in $H^*(F_{\mathbb{Z}_k^m}( \mathbb{C}^m, \Gamma))_{\theta\vee \psi}$, given by 
	\begin{equation}\label{prod1}
		x\cup y =\begin{cases}
			\text{sgn}(|\theta,\psi|)(a\cdot b) \otimes \sum_{ij}k_is_j\eta_i\vee\nu_j  & \text{ for independent } \theta,\psi\\
			0 & \text{ otherwise }	
	\end{cases}\end{equation}
	where $a\in A^*(L(\Gamma))_{\pi(\theta)},b \in A^*(L(\Gamma))_{\pi(\psi)},\sum k_i\eta_i\in \text{BCp}(\theta), \sum s_i\nu_i \in \text{BCp}(\psi)$, $\vee$ is the join operation in $L_k^m(\Gamma)$, $a\cdot b$ is the product in OS-algebra $A^*(L(\Gamma))$ and $|\theta,\psi|$ is the permutation in Definition \ref{permu}.
\end{thm}

\vskip .5em
\subsection{Cohomology ring of $H^*(F_{\mathbb{Z}_2^m}(\mathbb{R}^m,\Gamma))$}~
\vskip .5em

In the case of real number, the standard action $\mathbb{Z}_2^m \curvearrowright \mathbb{R}^m$ is real part of standard $\mathbb{Z}_2^m \curvearrowright \mathbb{C}^m$, similarly, define "chromatic orbit configuration space"
$$F_{\mathbb{Z}_2^m}(\mathbb{R}^m,\Gamma)= \{(x_1,x_2,...,x_n)\in \mathbb{R}^m \times...\times \mathbb{R}^m | Gx_i \cap Gx_j = \emptyset \text{ for } (ij)\in E(\Gamma)\}$$
$$V_{(e,g)}= \{(x_1,...,x_n)\in \mathbb{R}^m \times...\times \mathbb{R}^m| gx_i=x_j \text{ for } e=(ij)\}$$  and $\mathcal{A}=\{V_{(e,g)}|(e,g)\in E(\Gamma)\times \mathbb{Z}_2^m\}$ be a subspace arrangement in $\mathbb{R}^{mn}$ with intersection lattice $\mathcal{P}$ and complement $\mathcal{M}(\mathcal{A})=F_{\mathbb{Z}_2^m}(\mathbb{R}^m,\Gamma)$. This arrangement $\mathcal{A}$
is not a ($\geq 2$)-arrangement, the cohomology ring $H^*(F_{\mathbb{Z}_2^m}(\mathbb{R}^m,\Gamma))$ is described by \cite[Theorem 7.5]{DeLongueville2001} up to an "error term". 

More precisely, there is a $\mathcal{P}$-filtration $\{F_x| x\in \mathcal{P}\} $ on $H^*(\mathcal{M}(\mathcal{A}))$ satisfying $$x\leq y \Rightarrow F_x H^*(\mathcal{M}(\mathcal{A}))\subseteq F_y H^*(\mathcal{M}(\mathcal{A}))$$ $$a \in   F_x, b \in F_y \Rightarrow a\cup b \in F_{x\vee y}$$ The associated $\mathcal{P}$-graded ring $Gr = \bigoplus_{u\in \mathcal{P}} Gr_u$ is defined as $$Gr_u = F_u/\sum_{v<u} F_v$$ with multiplicative structure $Gr_x\otimes Gr_y \xrightarrow{\cdot} Gr_{x\vee y}$ given by 
\begin{equation*}
	[a]\cdot [b] =\begin{cases}
		[a\cup b] & \text{ for } x+y=M\\
		0 & \text{ otherwise }	
\end{cases}\end{equation*}
there is an isomorphism of ring $Gr(H^*(\mathcal{M}(\mathcal{A})))\cong \bigoplus_{x\in \mathcal{P} }\text{Tor}^{\mathcal{P}}_{*}(\delta_x \mathbb{Z},\delta^{M}\mathbb{Z})$ where the ring structure on right side is given by the composition of cross product and $\text{Tor}_{i+j}^\vee(\star,\star)$ if $x+y=M$ and be $0$ otherwise, see \cite[Theorem 7.5]{DeLongueville2001} for details. Notice that, in the case of real number, the intersection lattice $\mathcal{P}$ is isomorphic with the intersection lattice in the case of complex number (with different real codimension), then the ring structure of $\bigoplus_{x\in \mathcal{P} }\text{Tor}^{\mathcal{P}}_{*}(\delta_x \mathbb{Z},\delta^{M}\mathbb{Z})$ discussed in last subsection give us the following theorem.
\begin{thm}\label{thm real}
	Let $m>1$ and using $\mathbb{Z}_2$ coefficient, there is a filtration of $H^*(F_{\mathbb{Z}_2^m}( \mathbb{R}^m, \Gamma))$ such that the associated $Gr(H^*(F_{\mathbb{Z}_2^m}( \mathbb{R}^m, \Gamma)))$ is a $L_2^m(\Gamma)$-graded algebra, and every piece of $\theta \in L_2^m(\Gamma)$ is group $A^*(L(\Gamma))_{\pi(\theta)}\otimes \text{BCp}(\theta)$ , has degree $(m-1)r_b(\theta)$, the rank of this group is given by Equation \ref{eq of rank1} (let $k=2$). The product structure on $Gr(H^*(F_{\mathbb{Z}_2^m}( \mathbb{R}^m, \Gamma)))$ is given by Equation \ref{prod1}.
\end{thm}

\section{Appendix}
\subsection{Property of cellular form}

\begin{lem}\label{lem of form}
	If the cellular form $\varLambda(P,\mathcal{G})$ exists, then
	\begin{enumerate}
		
		\item $\bigoplus\limits_{y<x,r(y)=0}\mathcal{G}(y) \to \mathcal{G}(x)$ is always a surjection for $r(x)>0$.
		\item $\partial$ maps $\varLambda(P,\mathcal{G})_{x_1}$ isomorphically on $\ker(\bigoplus_{y<x_1}\mathcal{G}(y) \to \mathcal{G}(x_1))$ for any rank $1$ element $x_1$.
		\item Every piece $\varLambda(P,\mathcal{G})_x$ is free if $\mathcal{G}$ is free.
	\end{enumerate}
\end{lem}
\begin{proof}~
	\begin{enumerate}
		
		\item For any $x\in P,r(x)>0,y<x,r(y)=0$, we have following commutative diagram
		$$\xymatrix{
			\varLambda(P,\mathcal{G})_{y} \ar[r]^\cong\ar[d]& 	\mathcal{G}(y)\ar[d]\\
			H_0(\varLambda(P,\mathcal{G})_{\leq x}) \ar[r]^-\cong & \mathcal{G}(x)
		}$$
		where the first vertical arrow is induced by inclusion of chain $\varLambda(P,\mathcal{G})_{\leq y} \hookrightarrow \varLambda(P,\mathcal{G})_{\leq x}$ and the second vertical arrow is extension map of $\mathcal{G}$,
		then the following diagram 
		$$\xymatrix{
		\bigoplus\limits_{y<x,r(y)=0}\varLambda(P,\mathcal{G})_{y} \ar[r]^-\cong\ar[d]& 	\bigoplus\limits_{y<x,r(y)=0}\mathcal{G}(y)\ar[d]\\
		H_0(\varLambda(P,\mathcal{G})_{\leq x}) \ar[r]^-\cong & \mathcal{G}(x)
		}$$
		is also commutative, where the first arrow is quotient map from degree zero chain of $\varLambda(P,\mathcal{G})_{\leq x}$ to degree zero homology of $\varLambda(P,\mathcal{G})_{\leq x}$ 
		\item In above diagram, let $x$ be a rank $1$ element, then $\partial$ maps $\varLambda(P,\mathcal{G})_{x}$ isomorphically on the kernel of first vertical arrow since $H_1(\varLambda(P,\mathcal{G})_{\leq x})=0$ by definition of cellular form.
		\item For rank zero $x_0$, $\varLambda(P,\mathcal{G})_{x_0}=\mathcal{G}(x_0)$ is free. For rank $1$ element $x_1$, $\varLambda(P,\mathcal{G})_{x_1}\cong \ker(\bigoplus_{y<x_1}\mathcal{G}(y) \to \mathcal{G}(x_1))$ is also free. Then  $\varLambda(P,\mathcal{G})_x$ is free for any $x\in P, r(x)>1$ by an induction on the rank and the property of cellular form. 
	\end{enumerate}
\end{proof}

The following lemma is a little trick for cellular form. 

\begin{lem}\label{trick of cf}
	Assume that we have already known $(P,\mathcal{G})$ is cellular, $\varLambda$ is a $P$-graded differential  $\mathbb{Z}$-module, satisfying 
	\begin{enumerate}
		\item $\varLambda_{x_0}=\mathcal{G}(x_0)$ for any rank $0$ element $x_0$.
		\item $\varLambda_{x_1} = \ker(\bigoplus_{y<x_1}\mathcal{G}(y) \to \mathcal{G}(x_1))$ and differential $\partial$ on $\varLambda_{x_1}$ is the embedding $\varLambda_{x_1} \hookrightarrow \bigoplus_{y<x_1}\mathcal{G}(y)$ for any rank $1$ element $x_1$.
		\item $0\to \varLambda_{x} \xrightarrow{\partial} \bigoplus_{y \lessdot x}\varLambda_y \xrightarrow{\partial} \bigoplus_{y < x,r(y)=r(x)-2}\varLambda_y$ is exact for $r(x)>1$. 
	\end{enumerate}
Then $\varLambda$ is cellular form of $(P,\mathcal{G})$.
\end{lem}
\begin{proof}
	Let $\varLambda(P,\mathcal{G})$ be the cellular form of $(P,\mathcal{G})$, we are going to define a map $f : \varLambda \to \varLambda(P,\mathcal{G})$ of differential module such that $f$ maps $\varLambda_x$ to $\varLambda(P,\mathcal{G})_x$ for any $x\in P$. $\bigoplus_{r(x)\leq 1}\varLambda_x$ has already been isomorphic with   $\bigoplus_{r(x)\leq 1}\varLambda(P,\mathcal{G})_x$ (as differential module) by our assumption of $\varLambda$ and Lemma \ref{lem of form}. For $r(x)>1$, we define $f : \varLambda_x \to \varLambda(P,\mathcal{G})_x$ inductively by $f = \partial^{-1}f\partial$, then $\varLambda$ is isomorphic with $\varLambda(P,\mathcal{G})$ as $P$-graded differential module. 
\end{proof}
\vskip .5em
\subsection{Construct cellular form}~
\vskip .5em
For a given pair $(P,\mathcal{G})$, we can check the existence of cellular form $\varLambda(P,\mathcal{G})$ (and construct it if it exists) in polynomial time by following algorithm.
\begin{enumerate}
	\item If $\bigoplus\limits_{y<x,r(y)=0}\mathcal{G}(y) \to \mathcal{G}(x)$ is always a surjection for any $x\in P, r(x)>0$, then go to step (2), otherwise, $\varLambda(P,\mathcal{G})$ do not exists and exit.
	\item Let $\varLambda(P,\mathcal{G})_{x_0} = \mathcal{G}(x_0)$ for any rank zero $x_0$, $\varLambda(P,\mathcal{G})_{x_1}=ker(\bigoplus_{y<x_1}\mathcal{G}(y) \to \mathcal{G}(x_1))$ for any rank $1$ element $x_1$ and $\partial$ on $\varLambda(P,\mathcal{G})_{x_1}$ be the embedding $ \varLambda(P,\mathcal{G})_{x_1} \hookrightarrow \bigoplus_{y<x_1}\mathcal{G}(y)$. If there is some element $x\in P,r(x)>1$, then go to step (3), otherwise, $\varLambda(P,\mathcal{G})$ is cellular form of $(P,\mathcal{G})$ and exit.
	\item For any $x\in P, r(x)>1$, assume $\varLambda(P,\mathcal{G})$ has been defined on elements with lower rank than $r(x)$. If $$\ker(\bigoplus\limits_{y<x,r(y)=0}\mathcal{G}(y) \to \mathcal{G}(x))=\sum_{x_1<x, r(x_1)=1}\ker(\bigoplus\limits_{y<x_1,r(y)=0}\mathcal{G}(y) \to \mathcal{G}(x_1))$$
	and positive degree homology of $\varLambda(P,\mathcal{G})_{<x}$ is trivial, then define $\varLambda(P,\mathcal{G})_x$ be the kernel of $\partial$ on $\bigoplus_{y\lessdot x}\varLambda(P,\mathcal{G})_y$ and the $\partial$ on $\varLambda(P,\mathcal{G})_x$ be the embedding $\varLambda(P,\mathcal{G})_x\hookrightarrow \bigoplus_{y\lessdot x}\varLambda(P,\mathcal{G})_y$ and repeat this step, otherwise, $(P,\mathcal{G})$ is not (homological) cellular.
\end{enumerate}
The proof is easy via definition of cellular form and Lemma \ref{lem of form}.


\begin{thebibliography}{10}

\bibitem{Artin1962}
M.~Artin.
\newblock {\em {Grothendieck topologies: Notes on a seminar}}.
\newblock Springer, 1962.

\bibitem{Bibby2018}
Christin Bibby and Nir Gadish.
\newblock {Combinatorics of Orbit Configuration Spaces}.
\newblock {\em International Mathematics Research Notices}, pages 1--34, dec
2020.

\bibitem{Casto2016}
Kevin Casto.
\newblock {$\mathrm{FI}_G$-modules, orbit configuration spaces, and complex
	reflection groups}.
\newblock arXiv:1608.06317, aug 2016.

\bibitem{Chen2020}
Junda Chen, Zhi L{\"{u}}, and Jie Wu.
\newblock {Cohomology ring of manifold arrangements}.
\newblock 	arXiv:2002.02666, feb 2020.

\bibitem{Chen2021}
Junda Chen, Zhi L{\"{u}}, and Jie Wu.
\newblock {Orbit configuration spaces of small covers and quasi-toric
	manifolds}.
\newblock {\em Science China Mathematics}, 64(1):167--196, jan 2021.



\bibitem{DeConcini1995}
C.~{De Concini} and C.~Procesi.
\newblock {Wonderful models of subspace arrangements}.
\newblock {\em Selecta Mathematica}, 1(3):459--494, dec 1995.

\bibitem{DeLongueville2001}
Mark de~Longueville and Carsten~A. Schultz.
\newblock {The cohomology rings of complements of subspace arrangements}.
\newblock {\em Mathematische Annalen}, 319(4):625--646, apr 2001.

\bibitem{Deligne2000}
P.~Deligne, M.~Goresky, and R.~MacPherson.
\newblock {L'alg{\`{e}}bre de cohomologie du compl{\'{e}}ment, dans un espace
affine, d'une famille finie de sous-espaces affines.}
\newblock {\em Michigan Mathematical Journal}, 48(1):121--136, jan 2000.

\bibitem{DENHAM2018}
G.~Denham and A.~I.~Suciu.
\newblock {Local systems on complements of arrangements of smooth, complex
	algebraic hypersurfaces}.
\newblock {\em Forum of Mathematics, Sigma}, 6:e6, may 2018.



\bibitem{Dimca2009}
Alexandru Dimca and Sergey Yuzvinsky.
\newblock {Lectures on Orlik-Solomon Algebras}.
\newblock In Fouad {El Zein}, Alexandru~I. Suciu, Meral Tosun, A.~Muhammed
Uludag, and Sergey Yuzvinsky, editors, {\em Arrangements, Local Systems and
Singularities}, volume 283 of {\em Progress in Mathematics}, pages 83--110.
Birkh{\"{a}}user Basel, Basel, 2009.

\bibitem{Everitt2015}
Brent Everitt and Paul Turner.
\newblock {Cellular cohomology of posets with local coefficients}.
\newblock {\em Journal of Algebra}, 439:134--158, 2015.

\bibitem{Feichtner2002}
Eva~Maria Feichtner and G{\"{u}}nter~M. Ziegler.
\newblock {On orbit configuration spaces of spheres}.
\newblock {\em Topology and its Applications}, 118(1-2):85--102, feb 2002.



\bibitem{Goresky1988}
Mark Goresky and Robert MacPherson.
\newblock {\em {Stratified Morse Theory}}.
\newblock Springer Berlin Heidelberg, Berlin, Heidelberg, 1988.

\bibitem{Jewell1994a}
Ken Jewell.
\newblock {Complements of sphere and subspace arrangements}.
\newblock {\em Topology and its Applications}, 56(3):199--214, apr 1994.

\bibitem{Jewell1994}
Ken Jewell, Peter Orlik, and Boris~Z. Shapiro.
\newblock {On the complements of affine subspace arrangements}.
\newblock {\em Topology and its Applications}, 56(3):215--233, apr 1994.

\bibitem{Kecerdasan1998}
Inventori Kecerdasan and Pelbagai Ikep.
\newblock {Chapter 1 Introduction to fibred category theory}.
\newblock In {\em Categorical logic and type theory}, pages 19--117. 1998.

\bibitem{Vassiliev1992}
V.~A. Vassiliev.
\newblock {\em {Complements of Discriminants of Smooth Maps: Topology and
Applications}}, volume~98 of {\em Translations of Mathematical Monographs}.
\newblock American Mathematical Society, Providence, Rhode Island, jun 1992.

\bibitem{Weibel1994}
Charles~A. Weibel.
\newblock {\em {An Introduction to Homological Algebra}}.
\newblock Cambridge University Press, Cambridge, apr 1994.

\bibitem{Xicotncatl}
M.~A. Xicotncatl.
\newblock {\em {Orbit Configuration spaces, infinitesimal braid relations in homology and
		equivariant loop spaces}}.
\newblock Ph.D. Thesis, The University of Rochester, 1997.

\bibitem{Yuzvinsky2001}
S~A Yuzvinsky.
\newblock {Orlik-Solomon algebras in algebra and topology}.
\newblock {\em Russian Mathematical Surveys}, 56(2):293--364, apr 2001.

\bibitem{Yuzvinsky2002}
Sergey Yuzvinsky.
\newblock {Small rational model of subspace complement}.
\newblock {\em Transactions of the American Mathematical Society},
354(5):1921--1945, jan 2002.

\end{thebibliography}


\end{document}